\documentclass[12pt]{amsart}

\usepackage{amstext}
\usepackage{amsthm}
\usepackage{amssymb}
\usepackage[a4paper, total={6in, 9in}]{geometry}
\usepackage{enumerate}

\newtheorem{theorem}{Theorem}[section]
\newtheorem{lemma}[theorem]{Lemma}
\newtheorem{proposition}[theorem]{Proposition}
\newtheorem{corollary}[theorem]{Corollary}
\newtheorem{remark}[theorem]{Remark}

\newtheorem{definition}[theorem]{Definition}

\DeclareMathOperator{\interior}{int}

\DeclareMathOperator{\tr}{tr}
\DeclareMathOperator{\sym}{Sym}
\DeclareMathOperator{\Cyl}{Cyl}
\DeclareMathOperator{\supp}{supp}

\DeclareMathOperator{\orthog}{O}

\title[Convexity estimates for moving hypersurfaces]{Convexity estimates for hypersurfaces moving by concave curvature functions}

\author{Stephen Lynch}
\address{Eberhard Karls Universit\"{a}t T\"{u}bingen \\ Auf der Morgenstelle 10\\ 72076 T\"{u}bingen\\ Germany}
\email{stephen.lynch@math.uni-tuebingen.de}

\begin{document}

\begin{abstract}
We study fully nonlinear geometric flows that deform strictly $k$-convex hypersurfaces in Euclidean space with pointwise normal speed given by a concave function of the principal curvatures. Specifically, the speeds we consider are obtained by performing a nonlinear interpolation between the mean and the $k$-harmonic mean of the principal curvatures. Our main result is a convexity estimate showing that, on compact solutions, regions of high curvature are approximately convex. In contrast to the mean curvature flow, the fully nonlinear flows considered here preserve $k$-convexity in a Riemannian background, and we show that the convexity estimate carries over to this setting as long as the ambient curvature satisfies a natural pinching condition.
\end{abstract}

\maketitle

\section{Introduction}

We consider evolution processes that deform smooth hypersurfaces in Euclidean space (or more generally a Riemannian manifold) with pointwise velocity determined by their extrinsic curvature. A fundamental example is the mean curvature flow, which arises as the $L^2$-gradient flow of the area functional, and deforms hypersurfaces with pointwise velocity equal to the mean curvature vector. In particular, a one-parameter family of smooth immersions $F:M\times[0,T) \to \mathbb{R}^{n+1}$ of a compact orientable $n$-manifold $M$ solves mean curvature flow if 
\[\partial_t F(x,t) = - H(x,t) \nu(x,t)\]
on $M\times[0,T)$, where $\nu$ is the outward-pointing unit normal and $H$ is the sum of the principal curvatures. In coordinates, the mean curvature flow equation reduces to a weakly parabolic quasilinear system of PDE for the immersion $F$, and there is a unique short-time solution of the flow starting from any smooth initial immersion of $M$. 

Compact solutions of mean curvature flow in a Euclidean background form finite-time singularities, and there is an ongoing program aimed at understanding the structure of singularities for different classes of solutions. Of particular interest are solutions that are (strictly) $k$-convex, meaning that the sum of the smallest $k$ principal curvatures is everywhere (positive) nonnegative. A tensor maximum principle argument applied to the second fundamental form shows that each of these conditions is preserved by the flow. At the endpoints, we have that convexity $(k=1)$ and mean-convexity $(k=n)$ are preserved. 

Huisken showed that convex solutions contract to round points \cite{Huisk84}, and Huisken-Sinestrari constructed a flow with surgeries (in which almost-singular regions are excised and `healed') for two-convex solutions of dimension $n \geq 3$ \cite{Huisk-Sin09} (see \cite{Bren-Huisk16} and \cite{Hasl-Klein_a} for extensions to $n=2$). This led to a topological classification of compact two-convex hypersurfaces in $\mathbb{R}^{n+1}$ for each $n\geq 3$. In \cite{Huisk-Sin09}, the control required to perform surgery is obtained using a package of a priori curvature pinching and gradient estimates, including the convexity estimate established earlier by Huisken-Sinestrari \cite{Huisk-Sin99a}, \cite{Huisk-Sin99}, and by White \cite{White03} using other methods. This estimate says that mean-convex solutions only form convex singularities (that is, the second fundamental form becomes nonnegative at points where the curvature is blowing up), and plays a similar role in mean-convex mean curvature flow as the Hamilton-Ivey estimate in three-dimensional Ricci flow \cite{Ivey}, \cite{Ham93}. 

In the present work we establish a priori convexity estimates for a new family of fully nonlinear hypersurface flows generalising the mean curvature flow. For each dimension $n \geq 2$ and integer $1 \leq k \leq n$ we define a family of speed functions $\gamma_{k,\rho}$, with domain the $k$-positive cone in $\mathbb{R}^n$, as follows:
\[\gamma_{k,\rho} (z) := \Bigg(\sum_{i_1 < \dots < i_k } \frac{\rho}{z_{i_1} + \dots + z_{i_k}} + \frac{1-\rho}{z_1 + \dots + z_n} \Bigg)^{-1}, \qquad \rho \in [0,1].\]
Each of these functions is concave, one-homogeneous and increasing in its arguments. We will be interested in families of strictly $k$-convex immersions that solve
\begin{equation}
\label{eq:CF_interp}
\partial_t F(x,t) = - G_{k,\rho}(x,t) \nu(x,t),
\end{equation}
where $G_{k,\rho}(x,t) := \gamma_{k,\rho}(\lambda(x,t))$. Here the components of $\lambda$ are the principal curvatures, which we always label so that $\lambda_1 \leq \dots \leq \lambda_n$. At the coordinate level, equation \eqref{eq:CF_interp} is a fully nonlinear, weakly parabolic system for $F$.  

For each fixed $k \leq n-1$ the family $\gamma_{k,\rho}$ constitutes a nonlinear interpolation between the mean when $\rho =0$ and the $k$-harmonic mean when $\rho=1$. In fact, $\gamma_{k,\rho}$ can be expressed as a weighted harmonic mean of these two functions:
\[\gamma_{k,\rho}(z) = (\rho \gamma_{k,1}(z)^{-1} + (1-\rho) \gamma_{k,0}(z)^{-1})^{-1}.\]
In case $k = n$, $\gamma_{k,\rho}$ is simply the mean. For the present work, a crucial property is that $\gamma_{k,\rho}$ vanishes at the boundary of the $k$-positive cone for every $\rho>0$. As a consequence, maximum principle arguments show that any solution of \eqref{eq:CF_interp} starting from a compact strictly $k$-convex hypersurface remains (uniformly) $k$-convex. 

In case $\rho$ is small and positive, the flow \eqref{eq:CF_interp} has certain favourable properties that allow us to prove the following convexity estimate (the $\rho = 0$ case is precisely the estimate in \cite{Huisk-Sin99}). We write $M_t$ for the hypersurface $F(M,t)$. 

\begin{theorem}
\label{thm:conv_est_intro}
Fix $n \geq 4$ and $3 \leq k \leq n-1$. Then there is a constant $\rho_0 = \rho_0(n,k)$ in $(0,1]$ with the following property. Suppose $\rho \in (0,\rho_0]$ and let 
\[F: M\times [0,T) \to \mathbb{R}^{n+1}\]
be a compact smooth solution of the flow \eqref{eq:CF_interp}. Then for each $\varepsilon >0$ there is a positive constant $C_\varepsilon = C_\varepsilon(n, k, \rho, M_0)$ such that 
\[\lambda_1 \geq -\varepsilon G_{k,\rho} - C_\varepsilon\]
holds on $ M\times [0,T)$. 
\end{theorem}

We leave out the case $k=2$, since a convexity estimate was established for the two-harmonic mean curvature flow in \cite{Bren-Huisk17}, and the same arguments imply a convexity estimate for \eqref{eq:CF_interp} whenever $k=2$ and $\rho \in (0,1]$. In this special situation, a convexity estimate can be deduced from a cylindrical estimate, which states that the quantity $H/G_{2,\rho}$ becomes optimal at a singularity of the flow. An analogous cylindrical estimate was established in \cite{Lang-Lyn19} for a large class of flows by concave speeds, but for solutions which are only $k$-convex, this estimate only implies that the second fundamental form becomes $(k-1)$-nonnegative at a singularity. When $k \geq 3$, it is not possible to prove a convexity estimate for \eqref{eq:CF_interp} only by comparing the mean curvature with the speed. 

The arguments used to establish convexity estimates for the mean curvature flow in \cite{Huisk-Sin99}, and for other flows by convex one-homogeneous speeds in \cite{ALM14}, are also not applicable in our setting. The problem is that a certain gradient term appearing in the evolution of the second fundamental form (see Section \ref{sec:evol_eqs}) has the right sign for controlling $\lambda_1$ from below when the speed is convex, but is unfavourable in this regard when the speed is concave. The presence of this term, which makes even the study of convex solutions moving by concave speeds subtle (see \cite{And07}, and also the discussion of this work in Remark \ref{rem:inverse-concavity}), gives rise to even more serious difficulties when we move to the $k$-convex setting. It is by introducing the parameter $\rho$ that we are able to overcome these difficulties and prove Theorem \ref{thm:conv_est_intro}.

Other important results prior to our own include convexity estimates for flows by certain non-homogeneous concave speeds \cite{Aless-Sin}, and for surfaces moving by a very large class of one-homogeneous speeds \cite{And-Lang-McCoy15}.

\subsection{Curved ambient spaces} One motivation for studying the $k$-harmonic mean curvature flow, and more generally \eqref{eq:CF_interp}, is that these flows preserve $k$-convexity even when the solution is immersed in a (compact) Riemannian background space. This is in stark contrast to the mean curvature flow, which in a general ambient space will preserve mean-convexity, but not $k$-convexity for any $k \leq n-1$. Andrews first studied the harmonic mean curvature flow in Riemannian background spaces in \cite{And94}, and showed that when the background sectional curvatures are nonnegative, the flow contracts any compact strictly convex initial hypersurface to a round point. 

Flows in a Riemannian background are also the focus in the above mentioned work \cite{Bren-Huisk17} by Brendle-Huisken, which considers the two-harmonic mean curvature flow in backgrounds satisfying a natural curvature pinching assumption. Combining their cylindrical/convexity estimate with a non-collapsing estimate for emebedded solutions due to Andrews-Langford-McCoy \cite{And-Lang-McCoy13}, Brendle-Huisken established curvature gradient estimates and were able to implement the surgery procedure developed in \cite{Huisk-Sin09}. This led to a far-reaching generalisation of the topological classification of two-convex embeddings in \cite{Huisk-Sin09}. In a separate paper, the convexity estimate in Theorem \ref{thm:conv_est_intro} will be used to prove curvature gradient estimates for embedded solutions of \eqref{eq:CF_interp}. 

There is also a direct analogue of Theorem \ref{thm:conv_est_intro} for solutions of \eqref{eq:CF_interp} immersed in ambient spaces satisfying a natural curvature pinching condition. Since all of the major difficulties in establishing this result already occur when the background is Euclidean, we focus on this case, and only discuss flows in more general manifolds in Section \ref{sec:curved_ambient}.

\subsection{Outline} Let us describe the structure of the paper and the arguments used to prove the convexity estimate. In Section \ref{sec:notation} we fix notation and state some preliminary results. Of particular importance is a general a priori pinching estimate. This gives conditions under which a function defined on a solution of a hypersurface flow must tend to zero at points where the speed becomes unbounded. The proof is by Stampacchia iteration and follows \cite{Huisk84} and \cite{Bren-Huisk17}. Section \ref{sec:notation} also contains algebraic estimates relating the derivatives of $\gamma_{k,\rho}$ to those of $\gamma_{k,1}$. These show in particular that, when $\rho$ is small, $\gamma_{k,\rho}$ is close to being linear on compact subsets of the $k$-positive cone. 

In Section \ref{sec:unif_parabolic} we begin studying solutions of \eqref{eq:CF_interp}, and show that for each $\rho >0$ the flow preserves uniform $k$-convexity. From this it follows that curvature quantities satisfy a uniformly parabolic equation along the flow. The estimates of this section also imply that every compact solution of \eqref{eq:CF_interp} becomes singular in finite time, and that singularity formation is characterised by blow-up of $G_{k,\rho}$. In Section \ref{sec:cyl_est} we derive from the general pinching theorem a cylindrical estimate showing that the ratio $H/G_{k,\rho}$ becomes optimal at a singularity. Using the identity 
\[\frac{H}{G_{k,\rho}} = \rho \frac{H}{G_{k,1}} + 1-\rho,\]
we conclude that the ratio $H/G_{k,1}$ also becomes optimal at a singularity, irrespective of the value of $\rho >0$. 

In Section \ref{sec:smallest_eig} we state an evolution equation for $\lambda_1$ (interpreted in the barrier sense) due to Andrews \cite{And07} and begin analysing the gradient terms appearing in this equation. There is a favourable term coming from the concavity of $\lambda_1$ as a function of the second fundamental form, and an unfavourable term coming from the concavity of $\gamma_{k,\rho}$. We eventually show that the former outweighs the latter in regions of sufficiently high curvature, provided $\rho$ is small relative to $n$ and $k$. This step makes crucial use of the observation that $H/G_{k,1}$ improves in regions of high curvature. 

In Section \ref{sec:pinching_quantity} we divide $\lambda_1$ by the quantity appearing in the cylindrical estimate and show that if $\rho$ is sufficiently small,
the evolution equation of the resulting quantity has the right properties for applying the general pinching estimate. With this done we finally choose $\rho_0$ and prove the convexity estimate in Section \ref{sec:conv_est}, before sketching the proof of a generalisation to flows in Riemannian manifolds in Section \ref{sec:curved_ambient}.

\begin{remark}
The techniques developed here can also be applied to other flows. For example, suppose $\Gamma$ is an open, symmetric, convex cone in $\mathbb{R}^n$ and let $\gamma: \Gamma \to \mathbb{R}$ be smooth, symmetric, positive, one-homogeneous, and concave. Suppose in addition that there is a continous extension of $\gamma$ to $\bar \Gamma$ which vanishes on $\partial \Gamma$, and for each $\rho \in (0,1]$ define 
\[\gamma_\rho(z) := ( \rho \gamma(z)^{-1} + (1-\rho)(z_1 + \dots + z_n)^{-1} )^{-1}, \qquad  z \in \Gamma.\]
Then, as long as $\rho$ is sufficiently small, compact solutions of the flow
\[\partial_t F(x,t) = - G_\rho(x,t)\nu(x,t), \qquad G_\rho(x,t) := \gamma_\rho(\lambda(x,t)),\]
satisfy a convexity estimate analogous to that in Theorem \ref{thm:conv_est_intro}. 
\end{remark}

\subsection*{Acknowledgements} The author is grateful to Gerhard Huisken for many interesting and helpful discussions regarding this paper, and to Mat Langford and Ben Andrews for sharing insights on their work.  

\section{Notation and preliminary results}
\label{sec:notation}

Let $M$ be a compact, orientable smooth manifold of dimension $n \geq 2$ and consider a smooth one-parameter family of orientable immersions $F: M\times[0,T) \to \mathbb{R}^{n+1}.$
At each fixed time the immersion $F_t := F(\cdot,t)$ induces a metric and second fundamental form on $M$, which we denote by $g$ and $A$ respectively. In coordinates these tensors have components
\[g_{pq} = \bigg \langle \frac{\partial F}{\partial x^p}, \frac{\partial F}{\partial x^q} \bigg \rangle, \qquad A_{pq} = - \bigg \langle \frac{\partial^2 F}{\partial x^p \partial x^q}, \nu \bigg \rangle,\]
where $\nu$ is the outward-pointing unit normal. With respect to any basis, we write $g^{pq}$ for the components of the inverse of $g_{pq}$. The eigenvalues of the Weingarten map $A^p_q = g^{pr}A_{rq}$ are the principal curvatures $\lambda = (\lambda_1, \dots, \lambda_n),$
which we always label so that 
\[\lambda_1 \leq \dots \leq \lambda_n.\] The mean curvature of the immersion is the sum of the principal curvatures, 
\[H := \lambda_1 + \dots +\lambda_n.\]
Note that with our sign convention round spheres have positive mean curvature. We also write
\[A^2_{pq} := A_p^r A_{rq}.\]
The measure induced on $M$ by $F_t$ is denoted $\mu_t$. At each fixed time we write $\nabla$ for the Levi-Civita connection on $M$ associated with $g$. We recall the Codazzi equations, which imply that $\nabla A$ is totally symmetric:
\[\nabla_p A_{qr} = \nabla_r A_{pq}.\]

We write $M_t$ for the immersed hypersurface $F_t(M)$. We view geometric quantities such as the mean curvature as being defined on $M\times[0,T)$, or on the slices $M_t$, as is convenient. We state a number of estimates where some constant is said to depend on $M_0$. This means that said constant depends on the geometric properties of $M_0$ viewed as a hypersurface. 

We will be interested in families of immersions which evolve according to an equation of the form 
\begin{equation}
\partial_t F(x,t) = - G(x,t) \nu(x,t) \label{eq:CF_general}
\end{equation}
where $G(x,t) = \gamma(\lambda(x,t))$ and $\gamma$ is some speed function defined on an open, symmetric, convex cone $\Gamma \subset  \mathbb{R}^n$. We restrict attention to speeds which are:
\begin{enumerate}[(i)]
\item smooth;
\item positive; \label{pos}
\item symmetric;
\item one-homogeneous, i.e.  $\gamma(s z) = s \gamma(z)$ for every $s > 0$.
\end{enumerate}
Let us call $\gamma$ an admissible speed if it has all of these properties. The simplest example of an admissible speed is the mean, 
\[\tr(z) := z_1 + \dots + z_n.\]
The $k$-harmonic mean and the functions $\gamma_{k,\rho}$ defined in the introduction are admissible speeds defined on the $k$-positive cone.

If $\Gamma$ and $\Gamma'$ are symmetric cones in $\mathbb{R}^n$ and
\[ \overline {\Gamma} \cap \{z \in \mathbb{R}^n : \tr(z) = 1\} \subset \interior \Gamma',\]
then we write $\Gamma \Subset \Gamma'$.

\subsection{Differentiating symmetric functions}

For a symmetric cone $\Gamma \in \mathbb{R}^{n+1}$, let us denote by $\sym(\Gamma)$ the set of symmetric $n\times n$-matrices with eigenvalues in $\Gamma$. If $\gamma : \Gamma \to \mathbb{R}$ is smooth and symmetric, it extends to a smooth function (which we also denote by $\gamma$) on $\sym(\Gamma)$ satisfying
\[\gamma(O ZO^{-1}) = \gamma(Z)\]
for every $Z \in \sym(\Gamma)$ and $O \in \orthog(n)$ (this follows from Glaeser's composition theorem \cite{Glaeser}). We write $\dot \gamma^{p}$ and $\ddot \gamma^{pq}$ for the first and second derivatives of $\gamma$ with respect to eigenvalues, so that for every $z \in \Gamma$ and $\xi \in \mathbb{R}^n$,
\[\frac{d}{ds} \Big|_{s = 0 } \gamma(z + s\xi)= \dot \gamma^p(z)\xi_p , \qquad \frac{d^2}{ds^2} \Big|_{s = 0 } \gamma(z + s\xi)= \ddot \gamma^{pq}(z)\xi_p\xi_q .\]
Similarly, we use $\dot \gamma^{pq}$ and $\ddot \gamma^{pq,rs}$ to denote derivatives with respect to matrix components, so that for each $Z \in \sym(\Gamma)$ and $B \in \sym(\mathbb{R}^n)$ there holds 
\[\frac{d}{ds} \Big|_{s = 0 } \gamma(Z + sB) = \dot \gamma^{pq}(z) B_{pq}, \qquad \frac{d^2}{ds^2} \Big|_{s = 0 }\gamma(Z + sB) = \ddot \gamma^{pq,rs}(Z) B_{pq} B_{rs}.\]
If $Z \in \sym (\Gamma)$ is a diagonal matrix with eigenvalues $z$ then for each $\xi \in \mathbb{R}^n$ we have
\begin{align*}
\dot \gamma^{pq}(Z)\xi_p \xi_q &= \dot \gamma^{p} (z) \xi_p^2.
\end{align*}
If in addition $z_1 < \dots <z_n$ then
\begin{align*}
\ddot \gamma^{pq,rs}(Z) B_{pq} B_{rs} &= \ddot \gamma^{pq} (z) B_{pp} B_{qq} + 2 \sum_{p > q} \frac{\dot \gamma^{p} (z) - \dot \gamma^{q} (z)}{z_p - z_q} |B_{pq}|^2
\end{align*}
for every $B \in \sym(\mathbb{R}^n)$ (this is proven in \cite{And07}). Hence, if $\gamma$ is increasing in each of its arguments, $\dot \gamma^{pq}$ is positive-definite. Moreover, if $\Gamma$ is convex, the second identity implies that $\gamma$ is concave as a function of $z \in \Gamma$ if and only if it is concave as a function of $Z \in \sym(\Gamma)$. One direction is trivial, and the other is implied by the following fact: if $\gamma$ is symmetric and concave on $\Gamma$ and $z \in \Gamma$ is such that $z_p > z_q$, there holds
\begin{equation}
\label{eq:concave_firstderivs}
\frac{\dot \gamma^p(z) - \dot \gamma^q(z)}{z_p - z_q} \leq 0.
\end{equation}

Suppose now that $F:M\times[0,T) \to \mathbb{R}^{n+1}$ is a family of hypersurfaces such that $\lambda(x,t) \in \Gamma$ for each $(x,t) \in M\times [0,T)$, and set $G(x,t) := \gamma(\lambda(x,t))$. On a spacetime neighbourhood about each point in $M\times[0,T)$ there is a smooth frame of tangent vectors to $M$ which is orthonormal with respect to the induced metric. With respect to this frame the second fundamental form can be viewed as a smooth field of symmetric matrices taking values in $\sym(\Gamma)$, and we can write $G(x,t) = \gamma(A(x,t))$, which shows in particular that $G$ is smooth on $M \times [0,T)$. Moreover, because of the $\orthog (n)$-invariance of $\gamma$, at each time the derivatives $\dot \gamma^{pq}(A)$ and $\ddot \gamma^{pq,rs}(A)$ are the components of tensor fields on $M$.

\subsection{Evolution equations}
\label{sec:evol_eqs}
For a smooth solution $F:M\times [0,T) \to \mathbb{R}^{n+1}$ of \eqref{eq:CF_general}, where $\gamma : \Gamma \to \mathbb{R}$ is an admissible speed, we have the following evolution equations for geometric quantities. These were derived by Huisken for mean curvature flow \cite{Huisk84} and by Andrews \cite{And94_euclid} in the fully nonlinear case. Here and throughout, we sum over repeated indices, and for the sake of simplicity all expressions involving indices are written with respect to an orthonormal frame.  The important first-order quantities are:
\begin{align*}
\partial_t g_{ij} &= - 2 G A_{ij}\\
\partial_t g^{ij} & = 2 G A^{ij}\\
\partial_t \nu &= \nabla G\\
\partial_t d\mu_t &= - H G d\mu_t.
\end{align*}
The speed satisfies a parabolic equation,
\[(\partial_t - \dot \gamma^{pq} \nabla_p \nabla_q) G =  \dot \gamma^{pq} A^2_{pq} G,\]
as does the second fundamental form,
\[(\partial_t - \dot \gamma^{pq} \nabla_p \nabla_q) A_j^i= \dot \gamma^{pq} A^2_{pq} A^i_j  + \ddot \gamma^{pq,rs} \nabla^i A_{pq} \nabla_j A_{rs}.\]
Taking the trace of the last equation yields the evolution of the mean curvature,
\[(\partial_t - \dot \gamma^{pq} \nabla_p \nabla_q)  H= \dot \gamma^{pq}  A^2_{pq} H  + \ddot \gamma^{pq,rs} \nabla^i A_{pq} \nabla_i A_{rs}.\]

\subsection{A general pinching estimate}

Like the mean curvature flow, the flows by concave nonlinear speeds studied in this paper have the property that compact solutions form finite-time singularities, at which the value of $G(x,t) = \gamma(\lambda(x,t))$ becomes unbounded. We are going to state a general pinching estimate, which gives conditions under which a function $u:M\times[0,T) \to \mathbb{R}^{n+1}$ necessarily tends to zero at points where $G$ is blowing up. To prove Theorem \ref{thm:conv_est_intro}, we will apply this result to a quantity built from the smallest eigenvalue of the second fundamental form.

Huisken established the first result of this kind for convex solutions of mean curvature flow in \cite{Huisk84}, and the technique (a Stampacchia iteration scheme using the Michael-Simon Sobolev inequality) has since been built upon and applied to non-convex solutions of mean curvature flow (see \cite{Huisk-Sin99}, \cite{Huisk-Sin99a}, \cite{Huisk-Sin09}, \cite{Brendle15}, \cite{Lang17}) and classes of fully nonlinear flows (see \cite{ALM14}, \cite{And-Lang-McCoy15}, \cite{And-Lang_cyl}, \cite{Bren-Huisk17}, \cite{Lang-Lyn19}). There have also been some extensions to high-codimension mean curvature flow \cite{Andrews-Baker}, \cite{Lynch-Nguyen}. A key step in all of these proofs is to establish a Poincar\'{e}-type inequality for functions on $M_t$ which are supported away from points where the geometry looks like that of a cylinder. This step is carried out in a simple and direct way in \cite{Bren-Huisk17}[Proposition 3.1], and their idea is built in to Theorem \ref{thm:Stamp}.

One feature which has not appeared explicitly in previous estimates of this kind is that all of the hypotheses only need to hold at high curvature scales, i.e., at points where $G$ is extremely large relative to the initial data. This turns out to be crucial in the proof of our convexity estimate.

For each $0 \leq m \leq n-1$ we write 
\[\Cyl_m := \{ (\underbrace{0,\dots,0}_{m \text{ entries}}, R, \dots, R) \in \mathbb{R}^n: R> 0 \}.\]
This is the set of possible eigenvalue $n$-tuples of the cylinders $\mathbb{R}^m \times \partial B_R^{n-m}(0)$ in $\mathbb{R}^{n+1}$. Define also
\[\Cyl := \bigcup_{0 \leq m \leq n-1} \Cyl_m.\]

\begin{theorem}
\label{thm:Stamp}
Let $\gamma:\Gamma \to \mathbb{R}$ be an admissible speed and suppose \[F:M\times [0,T) \to \mathbb{R}^{n+1}\] is a smooth family of immersions satisfying
\[\partial_t  F(x,t) = - G(x,t) \nu(x,t),\]
where $G(x,t) := \gamma(\lambda(x,t))$. Define
$L := \sup_{M_0} G$.
Let $u : M\times[0,T) \to \mathbb{R}$ be a smooth function satisfying $u \leq C_0$, and suppose there is a constant $k_0 >0$ and a symmetric cone $\Gamma' \Subset \Gamma \setminus \Cyl$ such that 
\[\lambda(x,t) \in \Gamma' \qquad \text{\emph{for all}} \qquad  (x,t) \in \supp(u) \cap \supp(G - k_0L).\] 
Assume also that there are positive constants $C_1$, $C_2$, $C_3$, $C_4$ and $\delta \in (0,2]$ such that 
\begin{align}
\label{eq:stamp_evol}
(\partial_t - \dot \gamma^{pq} \nabla_p \nabla_q) u \leq   C_1 \frac{|\nabla u|^2}{u}- \frac{1}{C_2} u\frac{|\nabla A|^2}{G^2}  + C_3|A|^{2-\delta} +C_4
\end{align}
holds at every point in $\supp(u) \cap \supp(G - k_0 L)$, and set $C' := (C_0, C_1, C_2, C_3, C_4)$. Then there are constants $\sigma \in (0,1)$ and $C$ depending on
\[n,\gamma, \Gamma', k_0, C', \delta, L, \mu_0(M), T,\]
such that 
\[u(x,t) \leq CG(x,t)^{-\sigma}\]
for each $(x,t) \in M\times[0,T)$.
\end{theorem}
\begin{proof}
Combine the Stampacchia iteration procedure from \cite{Huisk84} with the Poincar\'{e}-type inequality in \cite{Bren-Huisk17}[Proposition 3.3]. Details can be found in the author's thesis \cite{Lynch_thesis}.
\end{proof}

\begin{remark}
\label{rem:weak_pinching}
The conclusion of the theorem remains true if, rather than being smooth, $u$ is only locally Lipschitz and satisfies the differential inequality \eqref{eq:stamp_evol} in the following weak sense: for every nonnegative locally Lipschitz function 
\[\varphi : M \times[0,T) \to \mathbb{R}\]
satisfying 
\[\supp(\varphi) \subset \supp(u) \cap \supp (G - k_0L),\]
the inequality
\begin{align}
\label{eq:weak_ineq}
 \int_{M_t}  \varphi \partial_t u\, d\mu_t &\leq   - \int_{M_t} \dot \gamma^{pq} \nabla_p u \nabla_q \varphi  \, d\mu_t  - \int_{M_t} \varphi \ddot \gamma^{rs,pq} \nabla_p A_{rs} \nabla_q u \, d\mu_t \notag\\
 & + C_1 \int_{M_t} \varphi  \frac{|\nabla u|^2}{u} \, d\mu_t- \frac{1}{C_2} \int_{M_t} \varphi u\frac{|\nabla A|^2}{G^2}   \, d\mu_t \notag \\
 &+ C_3 \int_{M_t} |A|^{2-\delta} \varphi \,d\mu_t + C_4 \int_{M_t} \varphi \,d\mu_t 
\end{align}
holds for almost every $t \in [0,T)$. If $u$ is smooth and satisfies \eqref{eq:stamp_evol} then this inequality is a consequence of the divergence theorem.
\end{remark}

\subsection{Algebraic properties of $\gamma_{k,\rho}$}
 
Suppose $n \geq 4$ and $3 \leq n \leq k-1$ are fixed, and write $\Gamma$ for the $k$-positive cone, i.e.,
\[\Gamma := \{ z \in \mathbb{R}^{n+1} :  z_{i_1} + \dots +z_{i_k} >0 \; \forall \; {1 \leq i_1 < \dots < i_k \leq n} \}.\]
We establish here some basic properties of the functions $\gamma_{k,\rho}$ defined in the introduction, each of which is a concave admissible speed on $\Gamma$. In fact, for each $\rho >0$ the function $\gamma_{k,\rho}$ is strictly concave in off-radial directions, by which we mean
\[\ddot \gamma_{k,\rho}^{pq}(z) \xi_p \xi_q \leq 0\]
for each $z \in \Gamma$ and $\xi \in \mathbb{R}^n$ with equality if and only if $\xi$ is a multiple of $z$. Notice that the Hessian of $\gamma$ must vanish in radial directions by one-homogeneity. 

For each $\alpha < \infty$ we define a convex cone 
\[\Gamma_\alpha:= \{z \in \Gamma: \tr(z) \leq \alpha \gamma_{k,1}(z)\},\]
and observe that since $\gamma_{k,1}$ is strictly concave in off-radial directions, for each $\alpha< \infty$,
\[\Gamma_\alpha \Subset \Gamma.\]
The family of smooth convex hypersurfaces $\partial \Gamma_\alpha\setminus\{0\}$ foliates $\Gamma \setminus \Cyl_0$ as $\alpha$ ranges over the interval
\[ \frac{n}{\gamma_{k,1}(1,\dots,1)}<\alpha< \infty.\]

Manipulating the definition of $\gamma_{k,\rho}$ we obtain the following inequalities relating it to $\gamma_{k,1}$ and the trace:
\begin{lemma}
\label{lem:zeroth-order_ests}
For each $\rho \in (0,1]$ and $z \in \Gamma$ there holds
\[\gamma_{k,1}(z) \leq \gamma_{k,\rho} (z) \leq \min\{\tr(z), \rho^{-1} \gamma_{k,1}(z)\}.\]
\end{lemma}

The first and second derivatives of $\gamma_{k,\rho}$ are related to those of $\gamma_{k,1}$ as follows. For each $Z \in \sym(\Gamma)$ and $B \in \sym(\mathbb{R}^n)$ we have:
\begin{align*}
\dot \gamma_{k,\rho}^{pq}(Z) &= \rho \frac{\gamma_{k,\rho}(Z)^2}{\gamma_{k,1}(Z)^2}  \dot \gamma_{k,1}^{pq}(Z)  + (1-\rho)\frac{\gamma_{k,\rho}(Z)^2}{\tr(Z)^2} \delta^{pq},
\end{align*}
and
\begin{align*}
\ddot \gamma_{k,\rho}^{pq, rs}(Z) B_{pq} B_{pq} &= \rho \frac{\gamma_{k,\rho}(Z)^2}{\gamma_{k,1}(Z)^2} \ddot \gamma_{k,1}^{pq,rs}(Z) B_{pq} B_{rs} \notag \\
&-2 \rho (1-\rho)\frac{\gamma_{k,\rho}(Z)^3}{\gamma_{k,1}(Z)  \tr(Z) } \bigg( \frac{\dot \gamma_{k,1}^{pq} (Z) B_{pq}}{\gamma_{k,1}(Z) } -\frac{\tr(B)}{\tr(Z) } \bigg)^2.
\end{align*}
The following two lemmata are obtained by combining these identities with Lemma \ref{lem:zeroth-order_ests}.
\begin{lemma}
\label{lem:first-order_ests}
For each $\rho \in (0, 1]$ and $Z \in \sym( \Gamma)$ there holds:
\begin{align*}
\dot \gamma_{k,\rho}^{pq}(Z) &\leq \min\bigg\{\frac{1}{\rho}, \frac{\tr(Z)^2}{\gamma_{k,1}(Z)^2} \bigg\}  \dot \gamma_{k,1}^{pq}(Z) +  \delta_{pq};\\
\dot \gamma_{k,\rho}^{pq}(Z)  &\geq \rho \dot \gamma_{k,1}^{pq}(Z) + (1-\rho)\frac{\gamma_{k,1}(Z)^2}{\tr(Z)^2} \delta_{pq}.
\end{align*}
\end{lemma}
\begin{lemma}
\label{lem:second-order_ests}
For each $\rho \in (0,1]$, $Z \in \sym(\Gamma)$ and $B \in \sym(\mathbb{R}^n)$ we have
\[\ddot \gamma_{k,\rho}^{pq,rs} (Z) B_{pq} B_{rs}\leq \rho \ddot \gamma_{k,1}^{pq,rs} (Z) B_{pq} B_{rs},\]
and
\begin{align*}
-\ddot \gamma_{k,\rho}^{pq,rs}&(Z) B_{pq} B_{rs} \leq \min\bigg\{\rho^{-2}, \rho \frac{\tr(Z)^3}{ \gamma_{k,1}(Z)^3}\bigg\} \bigg(-\ddot \gamma_{k,1}^{pq,rs}(Z)B_{pq} B_{rs} + C\frac{|B|^2}{\tr(Z)}\bigg),
\end{align*}
where $C = C(n,k)$.
\end{lemma}

Applying these results, we obtain the following uniform estimates on sets compactly contained away from the boundary of $\Gamma$. The second-derivative estimate implies that as $\rho \to 0$, the functions $\gamma_{k,\rho}$ converge to the mean in the $C^2$-norm on compact subsets of $\Gamma$.

\begin{lemma}
\label{lem:improved_deriv_bds}
Consider a symmetric cone $\Gamma' \Subset \Gamma$ and suppose $Z \in \sym (\Gamma')$. Then there is a positive constant $C = C(n,k,\Gamma')$ such that 
\[C^{-1} |\xi|^2 \leq \dot \gamma_{k,\rho}^{pq}(Z) \xi_p \xi_q \leq C |\xi|^2\]
for every $\xi \in \mathbb{R}^n$. In addition, there is a positive constant $C' = C'(n,k,\Gamma')$ such that 
\[- \ddot \gamma_{k,\rho}^{pq,rs}(Z) B_{pq} B_{rs} \leq C' \rho\frac{|B|^2}{\tr(Z)}\]
for every $B \in \sym(\mathbb{R}^n)$.
\end{lemma}

\begin{proof}
First observe that since $\Gamma' \Subset \Gamma$ the set
\[ \{Y \in \sym ({\Gamma'}) : \tr(Y) = 1\}\]
is precompact in $\sym(\Gamma)$. Consequently, since $\dot \gamma_{k,1}^{pq}$ is smooth and positive-definite in $\Gamma$, the constant
\[c:= \inf \{\dot \gamma_{k,1}^{pq}(Y) \xi_p \xi_q : Y \in \sym (\Gamma' ), \; \tr(Y) = 1, \; \xi \in \mathbb{R}^{n}, \; |\xi|=1\},\]
is well defined and positive. The one-homogeneity of $\gamma_{k,1}$ implies that 
\[\dot \gamma_{k,1}^{pq}(sY) = \dot \gamma_{k,1}^{pq}(Y)\]
for every $Y \in \sym(\Gamma)$ and $s >0$, so we conclude that 
\[\dot \gamma_{k,1}^{pq}(Y) \xi_p \xi_q = \dot \gamma_{k,1}^{pq}(\tr(Y)^{-1} Y) \xi_p \xi_q \geq c |\xi|^2\]
for every $Y \in \sym (\Gamma')$. The constant $c$ depends only on $n$, $k$ and $\Gamma'$, so this gives the desired lower bound for $\dot \gamma_{k,1}^{pq}(Z)$. 

The assumption $\Gamma' \Subset \Gamma$ also implies that
\[\gamma_{k,1}(Z)^2 \geq c(n,k,\Gamma') \tr(Z)^2,\]
where we have made $c$ smaller as necessary. Hence by Lemma \ref{lem:first-order_ests} there holds
\[\dot \gamma_{k,\rho}^{pq}(Z) \xi_p \xi_q \geq \rho c |\xi|^2 + (1-\rho) c |\xi|^2 = c |\xi|^2,\]
which gives the desired lower bound for $\gamma_{k,\rho}$.

The remaining inequalities are proven by very similar arguments, making use of Lemma \ref{lem:first-order_ests} and Lemma \ref{lem:second-order_ests}, and noting that by one-homogeneity
\[\ddot \gamma_{k,1}^{pq,rs}(sY)= s^{-1} \ddot \gamma_{k,1}^{pq,rs}(Y)\]
for each $Y \in \sym (\Gamma)$ and $s >0$. 
\end{proof}

\section{Uniform parabolicity}
\label{sec:unif_parabolic}

For the remainder of the paper let $n\geq 4$ and $3 \leq k \leq n-1$ be fixed. To ease notation we drop the index $k$, and for each $\rho\in(0,1]$, simply write $\gamma_\rho$ for the function 
\begin{align*}
\gamma_\rho(z) := \Bigg(\sum_{1 \leq i_1 < \dots < i_k \leq n} \frac{\rho}{z_{i_1} + \dots + z_{i_k}} + \frac{1-\rho}{z_1 + \dots + z_n} \Bigg)^{-1}.
\end{align*} 
Each of the functions $\gamma_\rho$ is a concave admissible speed on the $k$-positive cone $\Gamma$. To be precise,
\[\Gamma := \{ z \in \mathbb{R}^{n+1} :  z_{i_1} + \dots +z_{i_k} >0 \; \forall \; {1 \leq i_1 < \dots < i_k \leq n} \}.\]

Let us fix a $\rho \in (0,1]$ and consider a smooth family of strictly $k$-convex hypersurfaces,
\[F:M\times[0,T) \to \mathbb{R}^{n+1},\]
which we assume are evolving according to 
\begin{equation}
\label{eq:CF}
\partial_t F(x,t) = - G_\rho(x,t) \nu(x,t),
\end{equation}
where $G_\rho(x,t) := \gamma_\rho(\lambda(x,t)).$ Assume without loss of generality that $T$ is the maximal time of smooth existence for $F$ (for a proof of short-time existence of the flow starting from any compact strictly $k$-convex immersion we refer to \cite{Lang14}[Section 3.5]).

The results of this section follow \cite{And94_euclid}. We first observe that $\min_{M_t} G_\rho$ is bounded from below by its value at $t = 0$ and cannot remain bounded from above indefinitely, hence $T < \infty$. Note that this estimate yields a positive lower bound for the sum of the smallest $k$ principal curvatures along the flow, since 
\[\frac{\lambda_1 + \dots + \lambda_k}{\rho} \geq G_\rho\]
holds on $M\times [0,T)$.

\begin{lemma}
\label{lem:speed_lower}
For each $t \in [0,T)$ there holds 
\[\min_{M_t} G_{\rho} \geq \bigg( \Big(\min_{M_0} G_{\rho}\Big)^{-2} - c t \bigg)^{-\frac{1}{2}},\]
where $c = c(n,k,\rho)$. 
\end{lemma}
\begin{proof}
By the one-homogeneity of $\gamma_{\rho}$ and the Cauchy-Schwartz inequality we have 
\begin{align*}
G_{\rho}^2 &= (\dot \gamma_{\rho}^{p} \lambda_p)^2 \leq \bigg( \sum_p \dot  \gamma_{\rho}^p \bigg)( \dot \gamma_{\rho}^p \lambda_p^2)\leq C(n,k,\rho) \dot \gamma_{\rho}^{pq} A^2_{pq}.
\end{align*}
Here we have used the fact that the first derivatives of $\gamma_\rho$ are bounded over $\Gamma$ (see Lemma \ref{lem:first-order_ests}). Substituting this into the evolution equation for $G_{\rho}$ gives 
\[(\partial_t-\dot \gamma^{pq}_\rho \nabla_p \nabla_q )  G_{\rho} \geq \frac{1}{C} G_{\rho}^3,\]
so the parabolic maximum principle implies the desired inequality.
\end{proof}

Next we prove a scaling-invariant lower bound for $G_{\rho}$. From this estimate it follows that the ratio
\[\frac{\lambda_1 + \dots + \lambda_k}{H}\]
is bounded from below by a positive constant on $M\times[0,T)$. We say the solution is uniformly $k$-convex. 

\begin{lemma}
\label{lem:unif_parabolic}
For each $t \in [0,T)$ we have the inequality
\[\max_{M_t} \frac{H}{G_\rho} \leq \max_{M_0} \frac{H}{G_\rho}.\]
Moreover, there is a constant $\bar \alpha$ depending only on $n$, $k$ and $M_0$ such that 
\[ \max_{M_t} \frac{H}{G_1} \leq \bar \alpha \]
for each $t \in [0,T)$.
\end{lemma}
\begin{proof}
From the evolution equations for $G_{\rho}$ and $H$ we find that $u:= H/G_\rho$ satisfies 
\begin{align*}
(\partial_t  -\dot \gamma_{\rho}^{pq} \nabla_p \nabla_q ) u  =&  \frac{1}{G_{\rho}} \ddot \gamma_{\rho}^{pq,rs}\nabla^i A_{pq} \nabla_i A_{rs} + \frac{2}{G_{\rho}} \dot \gamma_{\rho}^{pq} \nabla_q G_{\rho} \nabla_p u.
\end{align*}
The first term on the right is nonpositive by the concavity of $\gamma_{k,1}$, so by the parabolic maximum principle we have 
\[\max_{M_t} u \leq \max_{M_0} u\]
for each $t \in [0,T)$. Next observe that since
\[\frac{H}{G_{\rho}} = \rho \frac{H}{G_{1}} + 1 -\rho \]
we have 
\[\max_{M_t} \frac{H}{G_1} \leq \max_{M_0} \frac{H}{G_1},\]
and the right-hand side depends only on $n$, $k$ and $M_0$.
\end{proof}

Recall that for each $\alpha < \infty$ we defined a convex cone $\Gamma_\alpha \Subset \Gamma$ by
\[\Gamma_\alpha := \{z \in \Gamma : \tr(z) \leq \alpha \gamma_1(z) \}.\]
To rephrase Lemma \ref{lem:unif_parabolic}, there is a uniform constant $\bar \alpha$ such that $\lambda \in \Gamma_{\bar \alpha}$ holds on $M\times[0,T)$. Combined with the first inequality of Lemma \ref{lem:improved_deriv_bds}, this implies that the operator $\dot \gamma_{\rho}^{pq} \nabla_p \nabla_q$ appearing in the evolution of the second fundamental form is uniformly elliptic along the flow, with ellipticity constant depending on $\bar \alpha$. 

As a consequence of the a priori estimates proven in this section, it can be shown that the maximal time of smooth existence $T$ is characterised by curvature blow-up:
\[\limsup_{t \to T} \max_{M_t} G_\rho = \infty.\]
This is established by writing the solution locally as a graph and applying the regularity theory for convex fully nonlinear parabolic PDE due to Evans \cite{Ev82} and Krylov \cite{Kryl82}. We refer to Section 4.3 of \cite{Lang14} for the details. Consequently, an argument similar to the proof of Lemma \ref{lem:speed_lower} shows that $T$ can be bounded in terms of $n$, $k$, $\rho$ and $M_0$. 

\section{A cylindrical estimate}
\label{sec:cyl_est}

We saw in Lemma \ref{lem:unif_parabolic} that the quantity $\max_{M_t}  H/ G_\rho$ is uniformly bounded from above for each $t \in [0,T)$. In this section we show that the ratio $H/G_\rho$ becomes optimal at a singularity. For each $\rho >0$ we write
\[ \alpha_{\rho} := \frac{n-k+1}{\gamma_\rho(\underbrace{0,\dots,0}_{k-1 \text{ entries}}, 1, \dots, 1)}.\]
That is, $\alpha_{\rho}$ is the value taken by $H/G_\rho$ on a cylinder $\mathbb{R}^{k-1} \times \partial B_1^{n-k+1}(0)$.

\begin{proposition}
\label{prop:cyl_est}
For each $\varepsilon >0$ there is a constant $C_\varepsilon$ depending only on $\varepsilon$, $n$, $k$, $\rho$ and $M_0$ such that 
\[H \leq (\alpha_{\rho} + \varepsilon)G_\rho + C_\varepsilon\]
on $M\times  [0,T)$. 
\end{proposition}

For the two-harmonic mean curvature flow this estimate was established in \cite{Bren-Huisk17}, and the arguments used there carry over to the present setting without major modifications. An adaptation of Proposition 3.6 in \cite{Bren-Huisk17} shows that for $z \in \Gamma$ satisfying
\[\tr(z) \leq \alpha_{\rho} \gamma_\rho(z)\]
there holds
\[\min_{1 \leq i_1 < \dots < i_{k-1} \leq n} z_{i_1} + \dots + z_{i_{k-1}} \geq 0,\]
with equality if and only if 
\[z = (\underbrace{0,\dots,0}_{k-1 \text{ entries}}, 1, \dots, 1).\]
Hence Proposition \ref{prop:cyl_est} implies that, at a singularity, the second fundamental form of $M_t$ either becomes strictly $(k-1)$-positive or approaches (up to rescaling) the second fundamental form of a cylinder $\mathbb{R}^{k-1} \times \partial B_1^{n-k+1}(0)$. This kind of estimate is usually referred to as a cylindrical estimate. 

Huisken and Sinestrari established the first cylindrical estimate in their work on two-convex solutions of mean curvature flow \cite{Huisk-Sin09}. Cylindrical estimates for a large class of flows by concave admissible speeds were established by Langford and the author in \cite{Lang-Lyn19}, and in fact Proposition \ref{prop:cyl_est} can be derived as a corollary of Theorem 1.1 in that paper. 

Let us show how Proposition \ref{prop:cyl_est} can be deduced from the general pinching estimate Theorem \ref{thm:Stamp}. The key observation (which will also play a role in the proof of the convexity estimate) is that, because of the concavity of the speed and the Codazzi equations, the gradient term appearing in the evolution equation for $H$ controls the full gradient of $A$. This was established by Andrews in \cite{And94_euclid}[Lemma 7.12] and later proven by a different argument in \cite{Bren-Huisk17}[Lemma 3.2]. We adapt the latter proof to establish:

\begin{lemma}
\label{lem:good_grad}
Let $\Gamma'$ be a symmetric cone satisfying $\Gamma' \Subset \Gamma$ and suppose $Z \in \sym(\Gamma')$. Then there is a constant $C = C(n,k,\Gamma')$ such that 
\[\sum_i \ddot \gamma_\rho^{pq,rs}(Z) T_{ipq} T_{irs} \leq - \frac{\rho}{C} \frac{|T|^2}{\tr(Z)}\]
for every totally symmetric $T$.
\end{lemma}

\begin{proof}
By the one-homogeneity of $\gamma_1$ we have 
\[\ddot \gamma_1^{pq,rs}(Z) Z_{pq} Z_{rs} = 0\]
so for any $B \in \sym(\mathbb{R}^n)$ there holds
\begin{align*}
\ddot \gamma_1^{pq,rs}(Z) B_{pq} B_{rs} &= \ddot \gamma_1^{pq,rs}(Z) \bigg(B_{pq} - \frac{\tr(B)}{\tr(Z)}Z_{pq} \bigg)\bigg(B_{rs} - \frac{\tr(B)}{\tr(Z)}Z_{rs} \bigg)\\
& +2\frac{\tr(B)}{\tr(Z)} \ddot \gamma_1^{pq,rs}(Z) B_{pq} Z_{rs}. 
\end{align*}
The last term on the right vanishes since $f(t) := \ddot \gamma_1^{pq,rs} (Z)(tB_{pq} + Z_{pq}) (tB_{rs} + Z_{rs}) $ is nonpositive and vanishes at $t = 0$, giving
\[0 = f'(0) = 2 \ddot \gamma_1^{pq,rs}(Z) B_{pq} Z_{rs}.\]
Using the assumption $\Gamma' \Subset \Gamma$ we deduce that
\[c_0 := \inf\{ -\ddot \gamma_1^{pq,rs} (Y) B_{pq}B_{rs} : Y \in \sym(\Gamma'), \; \tr(Y) = 1, \; \tr(B) = 0, \; |B| =1\}\]
is well defined. Moreover, since $\gamma_1$ is strictly concave in off-radial directions and the conditions $\tr(Y)=1$ and $\tr(B) = 0$ prevent $B$ from being proportional to $Y$, we have $c_0 > 0$. Hence by scaling
\[\ddot \gamma_1^{pq,rs} (Z) B_{pq} B_{rs} \leq - c_0 \frac{|B|^2}{\tr(Z)}\]
for all $Z \in \sym( \Gamma')$ and traceless $B$, and in particular, for every $B \in \sym(\mathbb{R}^n)$ we have
\[\ddot \gamma_1^{pq,rs}(Z) \bigg(B_{pq} - \frac{\tr(B)}{\tr(Z)}Z_{pq} \bigg)\bigg(B_{rs} - \frac{\tr(Z)}{\tr(B)}Z_{rs} \bigg) \leq - \frac{c_0}{\tr(Z)} \bigg|B - \frac{\tr(B)}{\tr(Z)} Z\bigg|^2.\]
Collecting these facts we obtain
\begin{equation}
\label{eq:good_grad_1}
\ddot \gamma_1^{pq,rs}(Z) B_{pq} B_{rs} \leq - \frac{c_0}{\tr(Z)} \bigg|B - \frac{\tr(B)}{\tr(Z)} Z\bigg|^2.
\end{equation}

Next observe that
\begin{align*}
 4\sum_{i,p,q} \bigg(T_{ipq}- \frac{\tr(T_i)}{\tr(Z)} Z_{pq} \bigg)^2 \geq \sum_{i,p,q} \bigg(T_{ipq}- \frac{\tr(T_i)}{\tr(Z)} Z_{pq} -T_{piq} + \frac{\tr(T_p)}{\tr(Z)} Z_{iq} \bigg)^2,
\end{align*}
so for $T$ totally symmetric,
\begin{align*}
 4\sum_{i,p,q} \bigg(T_{ipq}- \frac{\tr(T_i)}{\tr(Z)} Z_{pq} \bigg)^2 &\geq \sum_{i,p,q} \bigg(- \frac{\tr(T_i)}{\tr(Z)} Z_{pq} + \frac{\tr(T_p)}{\tr(Z)} Z_{iq} \bigg)^2\\
 &=\frac{2}{\tr(Z)^2} \sum_{i,p} (|Z|^2 \delta_{ip} - Z_{ip}^2)  \tr(T_i)\tr(T_p). 
\end{align*}
Let us define 
\[c_1 := \inf \{ (|Y|^2 \delta_{ip} - Y_{ip}^2 )\xi_i \xi_p: Y \in \sym(\Gamma'), \; \tr(Y) = 1, \; \xi \in \mathbb{R}^{n}, \; |\xi|=1\}\}\]
and observe that since $|Y|^2 \delta_{pq} - Y_{pq}^2$ is positive for every $Y \in \sym(\Gamma)$ (which we recall is the $k$-positive cone for some $k \leq n-1$) there holds $c_1 > 0$. Since $Z \in \sym(\Gamma')$ we obtain 
\begin{align*}
 2\sum_{i,p,q} \bigg(T_{ipq}- \frac{\tr(T_i)}{\tr(Z)} Z_{pq} \bigg)^2 \geq c_1 \sum_i \tr(T_i)^2. 
\end{align*}
On the other hand,
\begin{align*}
|T|^2 & = \sum_{i,p,q} \bigg(T_{ipq}- \frac{\tr(T_i)}{\tr(Z)} Z_{pq}  + \frac{\tr(T_i)}{\tr(Z)} Z_{pq}\bigg)^2\\
 &\leq 2\sum_{i,p,q} \bigg(T_{ipq}- \frac{\tr(T_i)}{\tr(Z)} Z_{pq} \bigg)^2 + 2 \frac{|Z|^2}{\tr(Z)^2} \sum_i \tr(T_i)^2,
\end{align*}
so by setting 
\[C_0 := \max\{|Y|^2 : Y \in \sym (\Gamma'), \; \tr(Y) = 1\},\]
we obtain
\begin{align*}
|T|^2 & \leq 2(1+ 2c_1^{-1} C_0)\sum_{i,p,q} \bigg(T_{ipq}- \frac{\tr(T_i)}{\tr(Z)} Z_{pq} \bigg)^2.
\end{align*}
Combining this with \eqref{eq:good_grad_1} gives 
\[\sum_i \ddot \gamma_1^{pq,rs}(Z) T_{ipq} T_{irs} \leq - \frac{1}{C} \frac{|T|^2}{\tr(Z)}\]
with $C = C(n,k,\Gamma')$. Appealing to Lemma \ref{lem:second-order_ests} we obtain
\[\sum_i \ddot \gamma_\rho^{pq,rs}(Z) T_{ipq} T_{irs} \leq - \frac{\rho}{C} \frac{|T|^2}{\tr(Z)}\]
for each $\rho \in (0,1]$.
\end{proof}

For each $\varepsilon >0$ we define a smooth function on $M\times[0,T)$ by
\[f_\varepsilon(x,t) := \frac{H(x,t) - (\alpha_{\rho} + \varepsilon) G_\rho(x,t)}{G_\rho(x,t)}.\]
With the previous lemma in hand we verify that $f_\varepsilon$ satisfies the hypotheses of Theorem \ref{thm:Stamp}, and so establish the cylindrical estimate.

\begin{proof}[Proof of Proposition \ref{prop:cyl_est}]
For each $\varepsilon >0$ we compute
\begin{align*}
(\partial_t - \dot \gamma_{\rho}^{pq} \nabla_p \nabla_q) f_\varepsilon &=  \frac{1}{G_\rho} \ddot \gamma_\rho^{pq,rs}\nabla^i A_{pq} \nabla_i A_{rs} + \frac{2}{G_\rho} \dot \gamma_\rho^{pq} \nabla_p G_\rho \nabla_q f_\varepsilon.
\end{align*}
By Lemma \ref{lem:unif_parabolic} we know that 
\[A \in \sym(\Gamma_{\bar \alpha})\]
holds on $M\times [0,T)$. Furthermore, by the Codazzi equations $\nabla A$ is totally symmetric, so we can apply the previous lemma with $\Gamma' = \Gamma_{\bar \alpha}$ to obtain a positive $C_0$ such that
\[\ddot \gamma_\rho^{pq,rs}\nabla^i A_{pq} \nabla_i A_{rs} \leq - \frac{\rho}{C_0} \frac{|\nabla A|^2}{H}\]
on $M\times [0,T)$. The cone $\Gamma_{\bar \alpha}$ is determined by $M_0$ via $\bar \alpha$, so $C_0 = C_0(n,k, M_0)$. We thus have 
\begin{align*}
(\partial_t - \dot \gamma_{\rho}^{pq} \nabla_p \nabla_q) f_\varepsilon &\leq  -\frac{\rho}{C_0} \frac{|\nabla A|^2}{G_\rho H}  + \frac{2}{G_\rho} \dot \gamma_\rho^{pq} \nabla_p G_\rho \nabla_q f_\varepsilon,
\end{align*}
and inserting the bounds (see Lemma \ref{lem:unif_parabolic} and Lemma \ref{lem:zeroth-order_ests})
\[f_\varepsilon \leq \bar \alpha, \qquad G_\rho \leq H,\]
we obtain
\begin{align*}
(\partial_t - \dot \gamma_{\rho}^{pq} \nabla_p \nabla_q) f_\varepsilon &\leq  -\frac{\rho}{\bar \alpha C_0} f_\varepsilon\frac{|\nabla A|^2}{H^2}  + \frac{2}{G_\rho} \dot \gamma_\rho^{pq} \nabla_p G_\rho \nabla_q f_\varepsilon
\end{align*}
on $M\times [0,T)$. 

Using $A \in \Gamma_{\bar \alpha}$ and Lemma \ref{lem:first-order_ests} there is a $C = C(n,k,M_0)$ such that 
\[\frac{2}{G_\rho} \dot \gamma_\rho^{pq} \nabla_p G_\rho \nabla_q f_\varepsilon \leq 2C \frac{|\nabla A|}{G_\rho} |\nabla f_\varepsilon|\]
on $M\times [0,T)$, so by Young's inequality we have 
\[\frac{2}{G_\rho} \dot \gamma_\rho^{pq} \nabla_p G_\rho \nabla_q f_\varepsilon \leq \frac{\rho}{2 \bar \alpha C_0} f_\varepsilon \frac{|\nabla A|^2}{G_\rho^2} + C_1(n,k,\rho, M_0)  \frac{|\nabla f_\varepsilon|^2}{f_\varepsilon}\]
on $\supp(f_\varepsilon)$. Consequently, at each point in $\supp(f_\varepsilon)$,
\begin{align*}
(\partial_t - \dot \gamma_{\rho}^{pq} \nabla_p \nabla_q) f_\varepsilon &\leq  -\frac{\rho}{2\bar \alpha C_0} f_\varepsilon\frac{|\nabla A|^2}{G_\rho^2}  + C_1  \frac{|\nabla f_\varepsilon|^2}{f_\varepsilon}.
\end{align*}
This inequality is in the form of \eqref{eq:stamp_evol}.

It remains to check that on $\supp(f_\varepsilon)$ the second fundamental form of the solution never coincides with that of a cylinder. Observe that, by the definition of $f_\varepsilon$ and Lemma \ref{lem:unif_parabolic}, on $\supp(f_\varepsilon)$ we have
\[\lambda \in \Gamma'' := \{z \in \Gamma : (\alpha_\rho + \varepsilon) \gamma_\rho(z) \leq \tr(z) \leq \bar \alpha \gamma_\rho(z)\}.\]
In the notation of Theorem \ref{thm:Stamp}, for each $m \leq k-1$ we have
\[\Cyl_m \subset \{z \in \Gamma :   \tr(z)\leq \alpha_\rho  \gamma_\rho(z)\},\]
whereas for $m \geq k$ there holds
\[\Cyl_m \subset \mathbb{R}^n \setminus \Gamma.\]
Putting these two facts together yields 
\[\Gamma'' \Subset \Gamma \setminus \Cyl\]

We may therefore invoke Theorem \ref{thm:Stamp} and apply Young's inequality to conclude that, for each $\varepsilon >0$, there is a positive $C_\varepsilon$ depending on $\varepsilon$, $n$, $k$, $\rho$ and $M_0$ such that
\[f_\varepsilon \leq \varepsilon + C_\varepsilon G_\rho\]
on $M\times [0,T)$. Note that while the constants coming from Theorem \ref{thm:Stamp} depend on the maximal time $T$, this quantity is controlled in terms of $n$, $k$, $\rho$ and $M_0$.
Rearranging gives 
\[H \leq (\alpha_\rho + 2\varepsilon) G_\rho + C_\varepsilon,\]
and since $\varepsilon $ can be made arbitrarily small this proves the claim.
\end{proof}

In addition to Proposition \ref{prop:cyl_est} we make use of the following direct corollary. Notice that the parameter $\rho$ appears only in the lower-order term $C_\varepsilon$. That is, regardless of the value of $\rho$, the principal curvatures of $M_t$ enter the cone $\Gamma_{\alpha_1}$ at a singularity. 

\begin{corollary}
\label{cor:cyl_est}
For every $\varepsilon >0$ there is a constant $C_\varepsilon$ depending only on $\varepsilon$, $n$, $k$, $\rho$ and $M_0$ such that 
\[H \leq (\alpha_1 + \varepsilon) G_1 + C_\varepsilon\]
on $M\times[0,T)$.
\end{corollary}

\begin{proof}
Substituting the identities
\[\frac{H}{G_\rho} = \rho \frac{H}{G_1} + 1-\rho, \qquad \alpha_{\rho} = \rho \alpha_1 + 1-\rho.\]
into the cylindrical estimate yields
\[\rho \frac{H}{G_1} \leq \rho \alpha_1 + \varepsilon + \frac{C_\varepsilon}{G_\rho},\]
or equivalently
\[H \leq (\alpha_1 + \rho^{-1} \varepsilon) G_1 + \rho^{-1} C_\varepsilon \frac{G_1}{G_\rho}.\]
We know that $G_1 \leq G_\rho$ by Lemma \ref{lem:zeroth-order_ests}, and $\varepsilon$ can be made arbitrarily small, so this estimate has the desired form. 
\end{proof}

\section{The smallest eigenvalue of $A$}
\label{sec:smallest_eig}

In this section we begin analysing the smallest eigenvalue of the second fundamental form of $M_t$. Although $\lambda_1$ is locally Lipschitz in both space in time, it may not be smooth at points of multiplicity. Due to this lack of regularity we interpret 
\[(\partial_t - \dot \gamma^{pq} \nabla_p \nabla_q )\lambda_1\]
in the barrier sense (following \cite{ALM14}, \cite{Brendle15}, \cite{Lang17}). 

\begin{definition}
Let $f:M\times [0,T) \to \mathbb{R}$ be locally Lipschitz continuous. Fix a point $(x_0,t_0) \in M \times (0,T)$. We say that a function $\varphi$ is a lower support for $f$ at $(x_0,t_0)$ if $\varphi$ is of class $C^2$ on the set $B_{g(t_0)}(x_0,r) \times [-r^2 +t_0,t_0]$ for some $r >0$, and there holds
\[f(x,t) \geq \varphi(x,t),\]
with equality at $(x_0,t_0)$. If the inequality is reversed then $\varphi$ is an upper support for $f$ at $(x_0,t_0)$. 
\end{definition}

Given a point $(x_0,t_0) \in M\times [0,T)$, let us say that $\{e_i\}_{i=1}^n \subset T_{x_0} M_{t_0}$ is a principal frame if $A(x_0,t_0)(e_i,e_i) = \lambda_i$ for each $1 \leq i \leq n$. Starting with the evolution equation for the second fundamental form a simple computation shows that, if $\varphi$ is a lower support for $\lambda_1$ at $(x_0,t_0)$, then in a principal frame at $(x_0,t_0)$ there holds 
\begin{align*}
(\partial_t - \dot \gamma_\rho^{pq} \nabla_p \nabla_q) \varphi &\geq \dot \gamma_\rho^{pq}  A^2_{pq} \varphi  + \ddot \gamma_\rho^{pq,rs} \nabla_1 A_{pq} \nabla_1 A_{rs}.
\end{align*}
Since $G_\rho$ is concave in the second fundamental form the gradient term has an unfavourable sign for controlling $\lambda_1$ from below using the maximum principle. However, it turns out that a stronger inequality is true. By computing much more carefully and fully exploiting the concavity of $\lambda_1$ as a function of $A$, Andrews could glean from the diffusion term an extra favourable gradient term \cite{And07} (see also \cite{Lang17}):

\begin{proposition}
\label{prop:lambda_1_evol}
Fix a point $(x_0,t_0) \in M\times (0,T)$ and let $\varphi$ be a lower support for $\lambda_1$ at $(x_0,t_0)$. Then in a principal frame at $(x_0,t_0)$ there holds 
\begin{align}
\label{eq:lambda_1_evol}
(\partial_t - \dot \gamma_\rho^{pq} \nabla_p \nabla_q) \varphi \geq& \dot \gamma_\rho^{pq}  A^2_{pq} \varphi  + \ddot \gamma_\rho^{pq,rs} \nabla_1 A_{pq} \nabla_1 A_{rs}\notag\\
&+ 2\sum_{\lambda_i > \lambda_1 }  \frac{1}{\lambda_i - \lambda_1} \dot \gamma_\rho^{pq}  \nabla_p A_{i1} \nabla_q A_{i1}.
\end{align}
\end{proposition}

\begin{remark}
\label{rem:inverse-concavity}
The same inequality holds for any admissible speed $\gamma$. When $M_t$ is convex the sum of the two gradient terms on the right-hand side is closely related to the Hessian of the function
\[\gamma_*(\lambda) := \gamma(\lambda_1^{-1}, \dots, \lambda_n^{-1})^{-1}.\]
In fact Andrews showed that if $\gamma$ is concave and $\gamma_*$ is concave on the positive cone (in which case $\gamma$ is said to be inverse-concave), then the flow with speed $\gamma$ preserves positive lower bounds on $\lambda_1$ and $\lambda_1 / H$. On the other hand, flows by speeds which are not inverse-concave will not, in general, preserve convexity \cite{AMY}.

It is not clear whether there is such an elegant characterisation of the gradient terms in \eqref{eq:lambda_1_evol} for non-convex solutions. Despite this, our convexity estimate makes essential use of the extra good term produced by Andrews' computation.  
\end{remark}

We are going to analyse the gradient terms in \eqref{eq:lambda_1_evol}. This is facilitated by the following elementary lemma. 

\begin{lemma}
\label{lem:lin_alg}
Fix $(x_0,t_0) \in M\times(0,T)$ and suppose $\lambda_1$ admits a lower support $\varphi$ at $(x_0,t_0)$. Then if $e_1$ and $e_2$ are two orthonormal vectors in $T_{x_0} M_{t_0}$ satisfying
\[A(x_0,t_0)(e_1, e_1) = A(x_0,t_0) (e_2, e_2) = \lambda_1(x_0,t_0),\]
there holds $\nabla A(e_1, e_2) = 0$ at $(x_0,t_0)$. 
\end{lemma}

\begin{proof}
Extend $e_1$ and $e_2$ to an orthonormal basis $\{e_i\}$ for $T_{x_0} M_{t_0}$, and then to an orthonormal frame in a neighbourhood of $x_0$ in $M_{t_0}$ by parallel transport with respect to the Levi-Civita connection. Then, computing at $x_0$, we have 
\begin{align*}
\nabla_k A(e_1, e_2) &= e_k( A_{12})  = \frac{1}{4} e_k (A(e_1 + e_2, e_1 + e_2) - A(e_1 - e_2, e_1 - e_2)).
\end{align*}
On the other hand, since $|e_1 + e_2| \equiv \sqrt{2}$, there holds 
\[A(e_1 + e_2, e_1+ e_2) \geq \sqrt{2} \lambda_1 \geq \sqrt{2} \varphi,\]
and by assumption, this inequality becomes an equality at $x_0$. From this we conclude that, at $x_0$, 
\[e_k (A(e_1 + e_2, e_1 + e_2) )= \sqrt{2} \nabla_k \varphi,\]
but the same argument shows that 
\[e_k( A(e_1 - e_2, e_1 - e_2) )= \sqrt{2} \nabla_k \varphi\]
also holds at $x_0$. Hence $\nabla_k A(e_1, e_2) = 0$ at $x_0$.    
\end{proof}

Combining the lemma with Proposition \ref{prop:lambda_1_evol}, we obtain the following estimate. By the improved cylindrical estimate Corollary \ref{cor:cyl_est}, if $\rho$ is small relative to $n$ and $k$, the second gradient term in \eqref{eq:lambda_1_grad_terms} is nonnegative at points of high curvature. This observation is the key to establishing our convexity estimate. We need to exploit this gradient term further in Proposition \ref{prop:f_evol_est}, and so delay placing any restrictions on $\rho$ until the proof of Theorem \ref{thm:conv_est_intro}, which can be found in Section \ref{sec:conv_est}.

\begin{lemma}
\label{lem:lambda_1_grad_terms}
Fix $(x_0,t_0) \in M\times (0,T)$ and suppose $\lambda(x_0,t_0) \in \Gamma' \Subset \Gamma.$
Then if $\varphi$ is a lower support for $\lambda_1$ at $(x_0,t_0)$, in a principal frame at $(x_0,t_0)$ we have
\begin{align}
\label{eq:lambda_1_grad_terms}
(\partial_t - \dot \gamma_\rho^{pq} \nabla_p \nabla_q)\varphi \geq \dot \gamma_\rho^{pq} A^2_{pq}\varphi  - C \rho \frac{|\nabla_1 \varphi|^2}{H}   + (C^{-1} - C \rho ) \sum_{p+q> 2} \frac{|\nabla_1 A_{pq}|^2}{H} ,
\end{align} 
where $C = C(n,k,\Gamma')$. 
\end{lemma}

\begin{proof}
We write $m$ for the dimension of the kernel of $A(x_0,t_0) - \lambda_1(x_0,t_0) g(x_0,t_0)$ so that $\lambda_i > \lambda_1$ if and only if $i\geq m+1$. Since we are working in a basis where $A(x_0,t_0)$ is diagonal, the estimate in Proposition \ref{prop:lambda_1_evol} can be simplified to 
\begin{align*}
(\partial_t - \dot \gamma_\rho^{pq} \nabla_p \nabla_q) \varphi \geq& \dot \gamma_\rho^{pq}  A^2_{pq} \varphi  + \ddot \gamma_\rho^{pq,rs} \nabla_1 A_{pq} \nabla_1 A_{rs}+ 2\sum_{\lambda_i > \lambda_1 }  \dot \gamma_\rho^{p}\frac{ |\nabla_p A_{i1}|^2}{\lambda_i - \lambda_1} .
\end{align*}
Since $\lambda(x_0,t_0) \in \Gamma'$, by Lemma \ref{lem:improved_deriv_bds} we can estimate $\dot \gamma_\rho^p (\lambda(x_0,t_0)) \geq c(n,k,\Gamma')$ and so obtain
\begin{align*}
2\sum_{\lambda_i > \lambda_1 }  \dot \gamma_\rho^{p}\frac{ |\nabla_p A_{i1}|^2}{\lambda_i - \lambda_1} =2\sum_{i\geq m+1 }  \dot \gamma_\rho^{p}\frac{ |\nabla_p A_{i1}|^2}{\lambda_i - \lambda_1}  &\geq 2c \sum_p \sum_{i \geq m+1 } \frac{|\nabla_p A_{i1}|^2}{\lambda_i - \lambda_1} .
\end{align*}
Since $\lambda_1 + \dots + \lambda_k >0$ and $k \leq n-1$ we have $\lambda_n < H$, and consequently
\[\lambda_i - \lambda_1 < \lambda_i +\lambda_2+ \dots + \lambda_k < kH\]
for each $i \geq m+1$. Collecting these inequalities we find that at $(x_0,t_0)$ there holds
\begin{align*}
(\partial_t - \dot \gamma_\rho^{pq} \nabla_p \nabla_q) \varphi \geq& \dot \gamma_\rho^{pq} A^2_{pq} \varphi  + \ddot \gamma_\rho^{pq,rs} \nabla_1 A_{pq} \nabla_1 A_{rs}+ \frac{2c}{k} \sum_p \sum_{i \geq m+1 } \frac{|\nabla_p A_{i1}|^2}{H} .
\end{align*}

By Lemma \ref{lem:lin_alg}, the definition of $m$, and the Codazzi equations, the tensor $\nabla_1 A$ has the following structure at $(x_0,t_0)$:
\begin{align*}
\nabla_1 A = & \nabla_1 A_{11} e^1 \otimes e^1 + \sum_{p \geq m+1} \nabla_1 A_{p1} e^p \otimes e^1 + \sum_{q \geq m+1} \nabla_1 A_{1q} e^1 \otimes e^q\\
&+ \sum_{p,q\geq m+1} \nabla_1 A_{pq} e^p \otimes e^q.
\end{align*}
Using the Codazzi equations again, we find that at $(x_0,t_0)$,
\begin{align*}
\sum_p \sum_{i \geq m+1 } |\nabla_p A_{i1}|^2 &= \sum_p \sum_{i \geq m+1} |\nabla_1 A_{ip}|^2\\
&=\sum_{i \geq m+1} |\nabla_1 A_{i1}|^2 +\sum_{i,p \geq m+1 } |\nabla_1 A_{ip}|^2 \\
&= \frac{1}{2} \sum_{p \geq m+1} |\nabla_1 A_{p1}|^2 + \frac{1}{2} \sum_{q \geq m+1} |\nabla_1 A_{1q}|^2 + \sum_{p,q \geq m+1} |\nabla_1 A_{pq}|^2\\
& \geq \frac{1}{2}\sum_{p+q > 2} |\nabla_1 A_{pq}|^2 .
\end{align*}
Hence at $(x_0,t_0)$ we have
 \begin{align}
 \label{eq:lambda_1_grad_a}
(\partial_t - \dot \gamma_\rho^{pq} \nabla_p \nabla_q) \varphi \geq& \dot \gamma_\rho^{pq}  A^2_{pq} \varphi  + \ddot \gamma_\rho^{pq,rs} \nabla_1 A_{pq} \nabla_1 A_{rs}+ \frac{c}{k}\sum_{p+q > 2} \frac{|\nabla_1 A_{pq}|^2}{H} .
\end{align}

Next we use Lemma \ref{lem:improved_deriv_bds} and the assumption $\lambda(x_0,t_0) \in \Gamma'$ to conclude that at $(x_0,t_0)$,
\begin{align*}
\ddot \gamma_\rho^{pq,rs} \nabla_1 A_{pq} \nabla_1 A_{rs} &\geq - C(n,k,\Gamma')\rho \frac{|\nabla_1 A|^2}{H}.
\end{align*}
Since $\nabla_1 A_{11}(x_0,t_0) = \nabla_1 \varphi(x_0,t_0)$ we can decompose the right-hand side as 
\[\ddot \gamma_\rho^{pq,rs} \nabla_1 A_{pq} \nabla_1 A_{rs} \geq - C\rho \frac{|\nabla_1 \varphi|^2}{H} - C\rho \sum_{p+q > 2} \frac{|\nabla_1 A_{pq}|^2}{H},\]
which gives the result upon substitution into \eqref{eq:lambda_1_grad_a}.
\end{proof}

\subsection{A pinching quantity}
\label{sec:pinching_quantity}

Rearranging the version of the cylindrical estimate from Corollary \ref{cor:cyl_est} we find that 
\[ 0 \leq G_1(x,t) - \frac{1}{\alpha_1 +\varepsilon} H(x,t) + \frac{C_\varepsilon}{\alpha_1+\varepsilon} \]
for each $(x,t) \in M\times [0,T)$. We set the parameter $\varepsilon$ equal to 
\[\varepsilon_0 := 10^{-10}\alpha_1\] 
in this estimate and set 
\[\mu := \frac{1}{2(1+ 10^{-10}) \alpha_1}, \qquad K:= \frac{C_{\varepsilon_0}}{(1+10^{-10})\alpha_1},\]
so that we may write
\[0 \leq G_1(x,t) - 2 \mu H(x,t) + K.\]
Since $G_\rho(x,t) \geq G_1(x,t)$ for every $\rho \in (0,1]$ we conclude that 
\[0 \leq G_\rho(x,t) - 2 \mu H(x,t) + K.\]

We will make use of the function $h(x,t):= G_\rho(x,t) - \mu H(x,t) +K$, which by construction satisfies
\[\mu H(x,t) \leq h(x,t) \leq G_\rho(x,t) + K\]
for every $(x,t) \in M \times[0,T)$. The constant $\mu$ depends only on $n$ and $k$, and
\[K = K(n, k,\rho, M_0).\]
The function $h$ evolves according to 
\[(\partial_t - \dot \gamma_{\rho}^{pq} \nabla_p \nabla_q )h = \dot \gamma_\rho^{pq} A^2_{pq} (h - K) - \mu \ddot \gamma_\rho^{pq,rs}  \nabla^i A_{pq} \nabla_i A_{rs}.\]
We are going to make use of the gradient term on the right to combat the gradient terms appearing in the evolution of $\lambda_1$. Here it will be crucial that the coefficient $\mu$ depends only on $n$ and $k$, since we will have to choose $\rho$ small depending on $\mu$. Otherwise, the choice $\varepsilon_0 = 10^{-10}\alpha_1$ is not special. 

For each $\eta \in(0,1]$, we define
\[f_\eta (x,t) = \frac{-\lambda_1(x,t) - \eta G_\rho(x,t)}{h(x,t)}.\]
The convexity estimate will be established by applying Theorem \ref{thm:Stamp} to these functions. Our immediate goal is to derive an evolution equation for $f_\eta$ and analyse the gradient terms appearing on the right-hand side. To do so, we employ the following elementary lemma.

\begin{lemma}
Let $\Gamma' \subset \{y \in \mathbb{R}^n : \tr(y) >0 \}$ be a symmetric, open, convex cone. Consider a function $\gamma:\Gamma' \to \mathbb{R}$ which is smooth, symmetric, one-homogeneous and concave, and satisfies
\[\gamma(1,\dots,1) >0.\]
Then if
$z \in \Gamma'$ is such that $z_1 \leq \dots \leq z_n$, there holds $\dot \gamma^1(z) \geq 0$. 
\end{lemma}

\begin{proof}
\label{lem:+_first_deriv}
Fix $z \in \Gamma'$ satisfying $z_1 \leq \dots \leq z_n$. Since $\gamma$ is concave, the super-level set 
\[S:= \{y \in \Gamma': \gamma(y) \geq \gamma(z)\}\]
is convex, and since $\gamma$ is symmetric, each of the vectors
\[(z_m , \dots, z_n, z_1, \dots, z_{m-1})\]
is in $S$. Taking the average, we get $\bar z \in S$, where 
\[\bar z := \frac{\tr(z)}{n} (1, \dots, 1).\]
By assumption all of the entries of $\bar z$ are positive. Therefore, since $\gamma(1,\dots,1) > 0$, for every $s \geq 1$ there holds 
\[\gamma(s\bar z) = s \gamma( \bar z) \geq \gamma(\bar z) \geq \gamma(z),\]
which means $s \bar z \in S$. Appealing again to the convexity of $S$, we find that for each $s \geq 1$, the line segment connecting $z$ with $s \bar z$ is contained in $S$. Taking a limit as $s \to \infty$ we conclude that the ray
\[\{z + s \bar z : s \in [0,\infty) \}\]
is contained in $S$. Another way to say this is that $\gamma(z) \leq \gamma(z + s \bar z)$ for every $s \geq 0$, so we have
\[0 \leq \frac{d}{ds}\bigg|_{s = 0} \gamma(z + s\bar z) = \dot \gamma^i(z) \bar z_i = \frac{\tr(z)}{n} \sum_{i =1}^n \dot \gamma^i (z).\]
Without loss of generality, it suffices to prove the claim for $z$ satisfying $z_1 < \dots < z_n$, since the general case then follows by approximation. Under this extra assumption, since $\gamma$ is symmetric and concave, by \eqref{eq:concave_firstderivs} we have 
\[\dot\gamma^j (z) \leq \dot \gamma^i(z)\] 
whenever $i < j$. Substituting this fact into the inequality above, we get
\[0 \leq \tr(z) \dot \gamma^1 (z) ,\]
and the claim follows. 
\end{proof}

With the lemma in hand we can establish the following estimate for the gradient terms in the evolution of $f_\eta$. 

\begin{proposition}
\label{prop:f_evol_est}
Let $(x_0,t_0)\in M\times (0,T)$ be such that $\lambda(x_0,t_0) \in \Gamma' \Subset \Gamma$ and let $\varphi$ be an upper support function for $f_\eta$ at the point $(x_0,t_0)$. Suppose in addition that $f_\eta(x_0,t_0) \geq 0$. Then in a principal frame at $(x_0,t_0)$ there holds 
\begin{align*}
(\partial_t - \dot \gamma_\rho^{pq}\nabla_p \nabla_q) \varphi & \leq K\dot \gamma_\rho^{pq} A^2_{pq}\frac{\varphi}{h}  + \mu\frac{\varphi}{h} \ddot \gamma_\rho^{pq,rs}  \nabla^i A_{pq} \nabla_i A_{rs}+ \frac{2}{h} \dot \gamma_\rho^{pq} \nabla_p \varphi \nabla_q h\\
&+ C\rho \frac{h}{H} |\nabla_1 \varphi|^2  -(C^{-1}- C \rho) \sum_{p+q> 2} \frac{|\nabla_1 A_{pq}|^2}{h H},
\end{align*}
where $C=C(n,k,\Gamma')$.
\end{proposition}
\begin{proof}
We first observe that the smooth function
\[\tilde \varphi(x,t) := -h(x,t) \varphi(x,t) - \eta G_\rho(x,t).\]
is a lower support for $\lambda_1$ at $(x_0,t_0)$, and
\[\varphi(x,t) = \frac{-\tilde \varphi(x,t) - \eta G_\rho(x,t)}{h(x,t)}.\]
Inserting the evolution equations for $G_\rho$ and $h$ into the identity
\begin{align*}
(\partial_t - \dot \gamma_\rho^{pq}\nabla_p \nabla_q) \varphi =& -\frac{1}{h} (\partial_t - \dot \gamma_\rho^{pq} \nabla_p \nabla_q) (\tilde \varphi+\eta G_\rho) - \frac{\varphi}{h} (\partial_t - \dot \gamma_\rho^{pq} \nabla_p \nabla_q) h\\
& + \frac{2}{h} \dot \gamma_\rho^{pq} \nabla_p h \nabla_q \varphi,
\end{align*}
we find that
\begin{align*}
(\partial_t - \dot \gamma_{\rho}^{pq} \nabla_p \nabla_q) \varphi =& -  \frac{1}{h} (\partial_t - \dot \gamma_{\rho}^{pq} \nabla_p \nabla_q) \tilde \varphi+K \dot \gamma_\rho^{pq} A^2_{pq} \frac{\varphi}{h} + \dot \gamma_\rho^{pq}A^2_{pq} \frac{\tilde \varphi}{h} \\
&  + \mu\frac{\varphi}{h} \ddot \gamma_\rho^{pq,rs}  \nabla^i A_{pq} \nabla_i A_{rs} + \frac{2}{h} \dot \gamma_\rho^{pq} \nabla_q h \nabla_p \varphi . 
\end{align*}
Applying Lemma \ref{lem:lambda_1_grad_terms} to $\tilde \varphi$ we find that at the point $(x_0,t_0)$ there holds
\begin{align*}
(\partial_t - \dot \gamma_{\rho}^{pq} \nabla_p \nabla_q) \tilde \varphi \geq \dot \gamma_\rho^{pq} A^2_{pq} \tilde \varphi  - C \rho \frac{|\nabla_1 \tilde \varphi|^2}{H}   + (C^{-1} - C \rho ) \sum_{p+q> 2} \frac{|\nabla_1 A_{pq}|^2}{H} ,
\end{align*} 
where $C = C(n,k,\Gamma')$, hence
\begin{align*}
(\partial_t - \dot \gamma_{\rho}^{pq} \nabla_p \nabla_q)\varphi & \leq K\dot \gamma_\rho^{pq} A^2_{pq}  \frac{\varphi}{h}  + \mu\frac{\varphi}{h} \ddot \gamma_\rho^{pq,rs}  \nabla^i A_{pq} \nabla_i A_{rs} + \frac{2}{h} \dot \gamma_\rho^{pq}  \nabla_p h \nabla_q \varphi\\
&+\frac{1}{h} \bigg( C \rho \frac{|\nabla_1 \tilde \varphi|^2}{H}   - (C^{-1} - C \rho ) \sum_{p+q> 2} \frac{|\nabla_1 A_{pq}|^2}{H} \bigg).
\end{align*}

We are going to decompose and then absorb part of the term $|\nabla_1 \tilde \varphi|^2$. At $(x_0,t_0)$ we compute
\begin{align*}
\nabla_1 \tilde \varphi &= -h \nabla_1 \varphi -\varphi \nabla_1 h - \eta \nabla_1 G_\rho\\
		&= - h \nabla_1 \varphi-\varphi (\dot \gamma^i_\rho - \mu)\nabla_1 A_{ii} - \eta \dot \gamma^i_\rho \nabla_1 A_{ii}\\
		&=-h \nabla_1 \varphi - (\eta \dot \gamma^1_\rho + \varphi(\dot \gamma^1_\rho - \mu)) \nabla_1 \tilde \varphi- \sum_{i \geq 2}^n (\eta \dot \gamma^i_\rho + \varphi(\dot \gamma^i_\rho - \mu))\nabla_1 A_{ii},
\end{align*} 
which we rearrange to obtain
\begin{align*}
(1 + \eta\dot \gamma_\rho^1 + \varphi(\dot \gamma_\rho^1 - \mu))\nabla_1 \tilde \varphi &=-h \nabla_1 \varphi - \sum_{i \geq 2}^n (\eta \dot \gamma_\rho^i + \varphi(\dot \gamma_\rho^i - \mu))\nabla_1 A_{ii}.
\end{align*}
 The function 
\[a(z):=  \gamma_\rho(z) - \mu \tr(z)\]
is symmetric, concave and one-homogeneous in $\Gamma$. Moreover, since $\gamma_\rho(z) \geq \gamma_1(z)$, $a(z)$ is positive whenever $\gamma_1(z) > \mu\tr(z)$, or equivalently when $z \in \interior (\Gamma_{1/\mu})$. By the definition of $\mu$, for each $0 \leq m \leq k-1$ we have 
\[(\underbrace{0,\dots,0}_{m \text{ entries}}, 1, \dots, 1) \in \interior (\Gamma_{1/\mu}),\]
so in particular 
\[a(1,\dots,1) >0.\]
We may therefore apply Lemma \ref{lem:+_first_deriv} with $\gamma = a$ to conclude that the quantity 
\[\frac{\partial a}{\partial z_1} (\lambda(x_0,t_0)) = \dot \gamma_\rho^1(\lambda(x_0,t_0)) - \mu \geq 0.\] 
We are assuming that
\[\varphi(x_0,t_0) = f_\eta(x_0,t_0) \geq 0,\]
so we have 
\[1 + \eta\dot \gamma_\rho^1 + \varphi(\dot \gamma_\rho^1 - \mu) \geq 1.\]
In particular, at the point $(x_0, t_0)$ there holds
\begin{align*}
\nabla_1 \tilde \varphi &= \frac{1}{1 + \eta\dot \gamma_\rho^1 + \varphi(\dot \gamma_\rho^1 - \mu)} \bigg(-h \nabla_1 \varphi - \sum_{i \geq 2}^n (\eta \dot \gamma_\rho^i + \varphi(\dot \gamma_\rho^i - \mu))\nabla_1 A_{ii}\bigg).
\end{align*}

Let us introduce the abbreviation 
\[\xi_i := \eta\dot \gamma_\rho^i(\lambda(x_0,t_0)) + \varphi(x_0,t_0)(\dot \gamma_\rho^i(x_0,t_0) - \mu),\]
so that we may write the last identity as
\begin{align*}
\nabla_1 \tilde \varphi &= -\frac{1}{1+\xi_1} h \nabla_1 \varphi - \sum_{i \geq 2}^n \frac{\xi_i}{1+\xi_1}\nabla_1 A_{ii}.
\end{align*}
Then since $\xi_1 \geq 0$ we can bound 
\begin{align*}
|\nabla_1 \tilde \varphi|^2 &\leq 2 h^2 |\nabla_1 \varphi|^2 + C(n) \sum_{i \geq 2}^n |\xi_i|^2|\nabla_1 A_{ii}|^2.
\end{align*}
There holds
\[|\xi_i|^2 \leq 2  \eta^2 |\dot \gamma_\rho^i|^2 + 4 \varphi^2(|\dot \gamma_\rho^i|^2 + \mu^2),\]
and $\eta \in (0,1]$ by definition. Since $\lambda(x_0,t_0) \in \Gamma'$, we can bound $\dot \gamma_\rho^{i}(\lambda(x_0,t_0))$ purely in terms of $n$, $k$ and $\Gamma'$ using Lemma \ref{lem:improved_deriv_bds}, and at $(x_0,t_0)$, 
\begin{align*}
0 \leq  \varphi = \frac{-\lambda_1 - \eta G_\rho}{h} \leq \frac{ \lambda_2 +\dots+ \lambda_k}{h} \leq (k-1) \frac{ H}{h} \leq \frac{k-1}{\mu}.
\end{align*}
Putting these facts together, we can bound $|\xi_i|$ purely in terms of $n$, $k$ and $\Gamma'$, hence
\begin{align*}
|\nabla_1 \tilde \varphi|^2 &\leq 2 h^2 |\nabla_1 \varphi|^2 + C(n,k,\Gamma') \sum_{i \geq 2}^n |\nabla_1 A_{ii}|^2.
\end{align*}
Substituting this estimate back in, we find that at $(x_0,t_0)$ there holds 
\begin{align*}
(\partial_t - \dot \gamma_{\rho}^{pq} \nabla_p \nabla_q)\varphi &\leq  K\dot \gamma_\rho^{pq} A^2_{pq}\frac{\varphi}{h}  + \mu\frac{\varphi}{h} \ddot \gamma_\rho^{pq,rs}  \nabla^i A_{pq} \nabla_i A_{rs} + \frac{2}{h} \dot \gamma^{pq}_\rho \nabla_p h  \nabla_q \varphi\\
&+C\rho\frac{h}{H}|\nabla_1 \varphi|^2 +  C \rho \sum_{i \geq 2}^n \frac{|\nabla_1 A_{ii}|^2}{hH}  - (C^{-1} - C \rho ) \sum_{p+q> 2} \frac{|\nabla_1 A_{pq}|^2}{hH} .
\end{align*}
Absorbing the second-last term into the last yields the desired estimate. 
\end{proof}

The last result completes our analysis of the gradient terms coming from the evolution of $\lambda_1$. Next we apply Lemma \ref{lem:good_grad} to extract a good term controlling $|\nabla A|^2$, and simplify somewhat.

\begin{lemma}
\label{lem:f_stamp_1}
Let $(x_0,t_0)\in M\times (0,T)$ be such that $\lambda(x_0,t_0) \in \Gamma' \Subset \Gamma$ and let $\varphi$ be an upper support function for $f_\eta$ at the point $(x_0,t_0)$. Suppose in addition that $f_\eta(x_0,t_0) > 0$. Then, in a principal frame at the point $(x_0,t_0)$ there holds 
\begin{align*}
(\partial_t - \dot \gamma_\rho^{pq} \nabla_p \nabla_q) \varphi & \leq CK|A| \varphi  - C^{-1} \rho \varphi \frac{|\nabla A|^2}{hH} + C(\rho^{-1} + KH^{-1} )  \frac{|\nabla \varphi|^2}{ \varphi }\\
&-(C^{-1}- C \rho) \sum_{p+q> 2} \frac{|\nabla_1 A_{pq}|^2}{h H},
\end{align*}
where $C = C(n,k,\Gamma')$. 
\end{lemma}

\begin{proof}
Let us write $C$ for a large constant depending only on $n$, $k$ and $\Gamma'$. By Proposition \ref{prop:f_evol_est} we have at $(x_0,t_0)$ the estimate
\begin{align*}
(\partial_t - \dot \gamma_\rho^{pq}\nabla_p \nabla_q) \varphi & \leq K\dot \gamma_\rho^{pq} A^2_{pq}\frac{\varphi}{h}  + \mu\frac{\varphi}{h} \ddot \gamma_\rho^{pq,rs}  \nabla^i A_{pq} \nabla_i A_{rs}+ \frac{2}{h} \dot \gamma_\rho^{pq} \nabla_p h\nabla_q \varphi \\
&+ C\rho \frac{h}{H} |\nabla_1 \varphi|^2  -(C^{-1}- C \rho) \sum_{p+q> 2} \frac{|\nabla_1 A_{pq}|^2}{h H}.
\end{align*}
Since $\lambda(x_0,t_0) \in \Gamma'$, by Lemma \ref{lem:improved_deriv_bds} we can estimate 
\[h^{-1} \dot \gamma_\rho^{pq} A^2_{pq}\leq C \mu^{-1} H^{-1} |A|^2 \leq C|A|,\]
which means that at $(x_0,t_0)$,
\[K\dot \gamma_\rho^{pq} A^2_{pq} \frac{\varphi}{h} \leq CK |A| \varphi.\]
Invoking Lemma \ref{lem:good_grad}, we can bound
\[\ddot \gamma_\rho^{pq,rs} \nabla^i A_{pq} \nabla_i A_{rs} \leq -c_0(n,k,\Gamma')\rho \frac{|\nabla A |^2}{H}.\]
Since $\varphi(x_0,t_0) > 0$ and $\dot \gamma_\rho^i(\lambda(x_0,t_0)) \leq C$, at the point $(x_0,t_0)$ we can use Young's inequality to estimate 
\begin{align*}
\frac{2}{h} \dot \gamma_\rho^{pq} \nabla_p h  \nabla_q \varphi&\leq  s C\varphi\frac{|\nabla h|^2}{h H} +  s^{-1} C\frac{H}{h} \frac{|\nabla \varphi|^2}{\varphi},
\end{align*}
where $s$ can be any positive number. At $(x_0,t_0)$ we have 
\[|\nabla h|^2 \leq 2|\nabla G_\rho|^2 + 2\mu^2 |\nabla A|^2  \leq C|\nabla A|^2,\]
and combining this with the previous inequality gives
\begin{align*}
\frac{2}{h} \dot \gamma_\rho^{pq} \nabla_p h \nabla_q \varphi &\leq  sC_0 \varphi\frac{|\nabla A|^2}{h H}+ s^{-1}C_0 \frac{H}{h} \frac{|\nabla \varphi|^2}{ \varphi } ,
\end{align*}
where $C_0 = C_0(n,k,\Gamma')$. Setting $s = \frac{c_0 \mu }{2C_0}\rho$ and putting all of this together, we get 
\begin{align*}
(\partial_t - \dot \gamma_\rho^{pq}\nabla_p \nabla_q) \varphi & \leq CK|A| \varphi  - \frac{c_0 \mu }{2} \rho \varphi \frac{|\nabla A|^2}{hH}+ C \rho^{-1} \frac{H}{h} \frac{|\nabla \varphi|^2}{ \varphi } + C\rho \frac{h}{H} |\nabla_1 \varphi|^2 \\
& -(C^{-1}- C \rho) \sum_{p+q> 2} \frac{|\nabla_1 A_{pq}|^2}{h H}.
\end{align*}
To finish, we use Lemma \ref{lem:zeroth-order_ests} to estimate
\[\frac{H}{h} \leq \mu^{-1}, \qquad \frac{h}{H}\leq \frac{G_\rho}{H} + \frac{K}{H} \leq 1 + \frac{K}{H},\]
and appeal to 
\[\varphi (x_0,t_0) \leq f_\eta(x_0,t_0) \leq (k-1) \mu^{-1}\]
to estimate
\begin{align*}
C \rho^{-1} \frac{H}{h} \frac{|\nabla \varphi|^2}{ \varphi } + C\rho \frac{h}{H} |\nabla_1 \varphi|^2 \leq C(\rho^{-1} + KH^{-1} )  \frac{|\nabla \varphi|^2}{ \varphi }.
\end{align*} 
\end{proof}

\section{Proof of the convexity estimate}
\label{sec:conv_est}

In the previous section we used barrier functions to interpret $(\partial_t - \dot \gamma_\rho^{pq}\nabla_p \nabla_q) \lambda_1$. To prove the convexity estimate we will instead need to work with a distributional interpretation. Following Brendle \cite{Brendle15} (see also \cite{Lang17}), we observe that $\lambda_1$ is a semiconcave function on $M\times [0,T)$ and apply Alexandrov's theorem. Using the characterisation 
\[\lambda_1(x,t) = \min_{e \in T_x M_t, \; |e| =1} A(x,t) (e,e),\]
it is possible to realise $\lambda_1$ locally as the infimum over a family of smooth functions which is compact in the $C^2$-norm. This is sufficient to conclude that $\lambda_1$ is locally semiconcave on $M\times [0,T)$, and from this fact it follows that for every choice of $\eta$ the function $f_\eta$ is locally semiconvex on $M\times[0,T)$. 

We discuss some properties of semiconvex functions on Riemannian manifolds in Appendix \ref{app:semiconvex}. In particular, by Alexandrov's theorem there is a set $Q$ of full measure in $M\times[0,T)$ on which $f_\eta$ is twice differentiable, and if $\varphi:M\times[0,T) \to \mathbb{R}$ is locally Lipschitz and nonnegative, then by Lemma \ref{lem:semiconvex_dist} we have
\begin{equation}
\label{eq:f_eta_IBP}
\int_{M_t} \varphi \dot \gamma^{pq}_\rho \nabla_p \nabla_q f_\eta \,d\mu_t \leq -\int_{M_t} \dot \gamma^{pq}_\rho \nabla_p f_\eta \nabla_q \varphi\,d\mu_t - \int_{M_t} \varphi \ddot \gamma_\rho^{rs,pq}\nabla_p A_{rs} \nabla_q f_\eta\, d\mu_t
\end{equation}
for almost every $t \in [0,T)$. Notice that if $f_\eta$ were smooth, this inequality would hold with equality by the divergence theorem. 

\begin{proposition}
There is a constant $\bar C = \bar C(n,k,\rho,M_0)$ such that if 
\[\varphi : M \times[0,T) \to \mathbb{R}\]
is nonnegative, locally Lipschitz and satisfies
\[\supp(\varphi) \subset \supp(f_\eta) \cap \supp(G_\rho - \bar C),\]
then for almost every $t \in [0,T)$ there holds 
\begin{align}
\label{eq:f_stamp_int}
\int_M \varphi \partial_t f_\eta \, d\mu_t &\leq - \int_{M_t} \dot \gamma_\rho^{pq} \nabla_p f_\eta \nabla_q \varphi\,d\mu_t - \int_{M_t} \varphi \ddot \gamma^{rs,pq}_\rho \nabla_p A_{rs} \nabla_q f_\eta \, d\mu_t \notag \\
& + C(\rho^{-1} + K) \int_{M_t} \varphi \frac{|\nabla f_\eta|^2}{f_\eta} \,d\mu_t - C^{-1} \rho \int_{M_t} \varphi f_\eta \frac{|\nabla A|^2}{H^2} \,d\mu_t \notag \\
& - (C^{-1} - C\rho) \sum_{p+q>2} \int_{M_t}\frac{|\nabla_1 A_{pq}|^2}{hH} \,d\mu_t + CK\int_{M_t} |A|\varphi \,d\mu_t,
\end{align}
where $C = C(n,k)$. 
\end{proposition}
\begin{proof}
First fix an arbitrary point $(x_0,t_0) \in Q$. Then since $f_\eta$ is twice differentiable at $(x_0,t_0)$ there exists an upper support $\phi$ for $f_\eta$ at $(x_0,t_0)$. Since $\phi(x,t) \geq f_\eta(x,t)$ with equality at $(x_0,t_0)$ we obtain 
\[(\partial_t - \dot \gamma^{pq}_\rho \nabla_p \nabla_q)f_\eta (x_0,t_0) \leq (\partial_t -\dot \gamma_\rho^{pq} \nabla_p \nabla_q) \phi(x_0,t_0), \qquad \nabla f_\eta(x_0,t_0) = \nabla \phi(x_0,t_0).\]
Substituting these facts into Lemma \ref{lem:f_stamp_1} we find that, provided $f_\eta(x_0,t_0)>0$, in a principal frame at $(x_0,t_0)$ there holds 
\begin{align*}
(\partial_t - \dot \gamma_\rho^{pq} \nabla_p \nabla_q) f_\eta & \leq CK|A| f_\eta - C^{-1} \rho f_\eta \frac{|\nabla A|^2}{hH} + C(\rho^{-1} + KH^{-1} )  \frac{|\nabla f_\eta|^2}{ f_\eta}\\
&-(C^{-1}- C \rho) \sum_{p+q> 2} \frac{|\nabla_1 A_{pq}|^2}{h H},
\end{align*}
where $C$ depends on $n$, $k$, and the distance from $\lambda(x_0,t_0)/H(x_0,t_0)$ to $\partial \Gamma$. 

Let us write 
\[\Gamma' = \Gamma_{\alpha_1 + \varepsilon_0} = \Gamma_{(1+10^{-10})\alpha_1},\]
and observe that $\Gamma'$ is completely determined by $n$ and $k$. By the cylindrical estimate in Corollary \ref{cor:cyl_est} we have 
\[H \leq (\alpha_1 + \varepsilon_0/2) G_1 + C_{\varepsilon_0/2}(n,k,\rho, M_0),\]
so by setting 
\[\bar C = 2 \varepsilon_0^{-1} \rho^{-1} C_{\varepsilon_0/2}\]
we ensure that $H(x,t) \leq (\alpha_1 +\varepsilon_0) G_1(x,t)$, or equivalently $\lambda(x,t) \in \Gamma'$, whenever $G_1(x,t) \geq \rho\bar C$. Then since $G_1 \geq \rho G_\rho$ we have 
\[\lambda(x,t) \in \Gamma' \qquad \text{for all} \qquad (x,t) \in \supp(G_\rho -\bar C).\]
In particular, if $(x_0,t_0) \in Q \cap \supp(f_\eta) \cap \supp(G_\rho - \bar C)$ then at $(x_0,t_0)$ there holds 
\begin{align*}
(\partial_t - \dot \gamma_\rho^{pq} \nabla_p \nabla_q) f_\eta & \leq CK|A| f_\eta - C^{-1} \rho f_\eta \frac{|\nabla A|^2}{hH} + C(\rho^{-1} + KH^{-1} )  \frac{|\nabla f_\eta|^2}{ f_\eta}\\
&-(C^{-1}- C \rho) \sum_{p+q> 2} \frac{|\nabla_1 A_{pq}|^2}{h H}
\end{align*}
with $C = C(n,k)$. 

Without loss of generality we may assume $\bar C \geq \max\{1,K\}$ so that the inequalities
\[1 \leq G_\rho \leq H, \qquad h \leq G_\rho + K \leq 2 G_\rho\]
hold on $\supp(G_\rho - \bar C)$. We can also bound $f_\eta \leq C(n,k)$, so at each point $(x_0, t_0) \in Q \cap \supp(f_\eta)\cap  \supp(G_\rho - \bar C)$ we have 
\begin{align*}
(\partial_t - \dot \gamma_\rho^{pq} \nabla_p \nabla_q) f_\eta & \leq CK|A| - C^{-1} \rho f_\eta \frac{|\nabla A|^2}{H^2} + C(\rho^{-1} + K )  \frac{|\nabla f_\eta|^2}{ f_\eta}\\
&-(C^{-1}- C \rho) \sum_{p+q> 2} \frac{|\nabla_1 A_{pq}|^2}{h H}
\end{align*}
with $C = C(n,k)$. Suppose $\varphi$ is nonnegative, locally Lipschitz and that 
\[\supp(\varphi) \subset \supp(f_\eta) \cap \supp(G_\rho - \bar C).\] 
For almost every $t \in [0,T)$ the set $Q \cap M_t$ has full measure in $M_t$, and on such a timeslice we can multiply the last inequality by $\varphi$ and integrate to obtain 
\begin{align*}
\int_M \varphi \partial_t f_\eta \, d\mu_t &\leq  \int_{M_t} \varphi \dot \gamma_\rho^{pq} \nabla_p \nabla_q f_\eta \,d\mu_t + C(\rho^{-1} + K) \int_{M_t} \varphi \frac{|\nabla f_\eta|^2}{f_\eta} \,d\mu_t \\
&- C^{-1} \rho \int_{M_t} \varphi f_\eta \frac{|\nabla A|^2}{H^2} \,d\mu_t  - (C^{-1} - C\rho) \sum_{p+q>2} \int_{M_t}\varphi \frac{|\nabla_1 A_{pq}|^2}{hH} \,d\mu_t\\
& + CK\int_{M_t} |A|\varphi\,d\mu_t ,
\end{align*}
with $C=C(n,k)$. The result now follows by applying the integration by parts inequality \eqref{eq:f_eta_IBP} to the first term on the right. 
\end{proof}

\begin{proof}[Proof of Theorem \ref{thm:conv_est_intro}]
We verify that for $\rho$ sufficiently small the functions $f_\eta$ satisfy the hypotheses of the pinching estimate in Theorem \ref{thm:Stamp}. Fix $\eta \in (0,1]$ and set
\[\Gamma' := \{z \in \Gamma: \min_{1 \leq i \leq n} z_i \leq - \eta \gamma_\rho(z)\} \cap \Gamma_{\alpha_1 +\varepsilon_0},\]
so that $\lambda(x,t) \in \Gamma'$ whenever $(x,t) \in \supp(f_\eta) \cap \supp(G_\rho -\bar C)$. Since $\Cyl$ is contained in the positive cone and $\Gamma_{\alpha_1 + \varepsilon_0}\Subset \Gamma$ we have 
\[\Gamma' \Subset \Gamma \setminus \Cyl.\]

By the last proposition there is a constant $\rho_0 = \rho_0(n,k)$ such that if $\rho \leq \rho_0$, then the solution $M_t$ has the following property: If $\varphi:M\times[0,T) \to \mathbb{R}$ is nonnegative, locally Lipschitz and satisfies 
\[\supp(\varphi) \subset \supp(f_\eta) \cap \supp(G_\rho - \bar C),\]
then for almost every $t \in [0,T)$ there holds 
\begin{align*}
\int_M \varphi \partial_t f_\eta \, d\mu_t &\leq - \int_{M_t} \dot \gamma_\rho^{pq} \nabla_p f_\eta \nabla_q \varphi\,d\mu_t - \int_{M_t} \varphi \ddot \gamma^{pq,rs}_\rho \nabla_p A_{rs} \nabla_q f_\eta \, d\mu_t \\
& + C(\rho^{-1} + K) \int_{M_t} \varphi \frac{|\nabla f_\eta|^2}{f_\eta} \,d\mu_t - C^{-1} \rho \int_{M_t} \varphi f_\eta \frac{|\nabla A|^2}{H^2} \,d\mu_t \\
&  + CK\int_{M_t} |A|\varphi  \,d\mu_t,
\end{align*}
where $C = C(n,k)$.

We may therefore apply Theorem \ref{thm:Stamp} (see also Remark \ref{rem:weak_pinching}) to conclude that for $\rho \leq \rho_0$ the estimate
\[f_\eta \leq \eta  + C_\eta(n,k,\rho,M_0) G_\rho^{-1}\]
holds on $M\times[0,T)$ for every $\eta \in (0,1]$. Unpacking this we find that 
\[\lambda_1 \geq -\eta G_\rho-\eta h - C_\eta G_\rho^{-1} h,\]
so since $h \leq G_\rho + K$ we have
\[\lambda_1 \geq - 2\eta G_\rho - \eta K - C_\eta - K C_\eta G_\rho^{-1}.\] 
Recalling from Lemma \ref{lem:unif_parabolic} that $G_\rho$ is bounded from below by its minimum over $M_0$, we have 
\[\lambda_1 \geq  - 2 \eta G_\rho- \tilde C_\eta,\]
with $\tilde C_\eta = \tilde C_\eta(n,k,\rho,M_0)$. Since $\eta \in (0,1]$ was arbitrary, the convexity estimate is proven. 
\end{proof}

\section{Curved ambient spaces}
\label{sec:curved_ambient}

As discussed in the introduction, Andrews showed that in a compact ambient manifold with nonnegative sectional curvatures, the harmonic mean curvature flow contracts any compact, strictly convex initial hypersurface to a round point \cite{And94}. In fact, the result Andrews proved is more general: if the sectional curvatures of the ambient metric are bounded from below by $-\kappa^2$, then the evolution of any compact initial hypersurface satisfying $\lambda_1 > \kappa$ by the speed 
\[\lambda \mapsto \bigg( \sum_i \frac{1}{\lambda_i - \kappa}\bigg)^{-1}\]
contracts to a round point. Similarly, Brendle-Huisken \cite{Bren-Huisk17} used the shifted speed function 
\[\lambda \mapsto \bigg( \sum_{ i<j} \frac{1}{\lambda_i +\lambda_j - 2\kappa}\bigg)^{-1}\]
to extend their results on two-convex embeddings to background spaces with some negative curvature. Following these authors, we describe in this section an analogue of Theorem \ref{thm:conv_est_intro} for flows in non-Euclidean background spaces.

Fix $n \geq 4$, let $(N,\bar g)$ be a compact Riemannian manifold and denote the Riemann curvature tensor by $\bar R$. Fix also $3 \leq k \leq n-1$ and let $\kappa \geq 0$ be such that 
\begin{align}
\label{eq:backg_pinch}
 \sum_{i=2}^{k+1}\bar R(e_i, e_1, e_i, e_1) \geq -k\kappa^2
\end{align}
for every collection of orthonormal tangent vector fields $\{e_i\}_{i=1}^{k+1}$ on $N$. For each $\rho >0$ let $\gamma_\rho$ be defined on the $k$-positive cone $\Gamma\subset \mathbb{R}^{n+1}$ as before, and set
\[\gamma_{\rho,\kappa} (z) := \gamma_{\rho}(z_1 - \kappa, \dots, z_n - \kappa).\]
Then we have the following theorem. 

\begin{theorem}
\label{thm:short_time_N}
Fix a smooth immersion $F_0 : M \to N$ satisfying 
\[\lambda_1 + \dots + \lambda_k > k \kappa.\]
Then for each $\rho>0$, there is a unique maximal smooth solution $F:M\times[0,T) \to N$ of the evolution equation
\begin{equation}
\label{eq:CF_shifted}
\partial_t F(x,t) = - G_{\rho,\kappa}(x,t) \nu(x,t)
\end{equation}
such that $F(\cdot,0) = F_0$, where $G_{\rho,\kappa}(x,t) := \gamma_{k,\rho}(\lambda(x,t))$. There are constants $c_0$ and $\bar \alpha$ depending only on $n$, $k$, $\kappa$, $\rho$, $M_0$ and $N$ such that the inequalities
\begin{equation}
\label{eq:pinch_pres_N}
G_{\rho,\kappa} \geq c_0, \qquad H\leq \bar \alpha G_{\rho, \kappa}
\end{equation}
hold on $M\times[0,T)$, and at the maximal time $T < \infty$ there holds 
\[\limsup_{t \to T} \max_{M_t} G_{\rho,\kappa} = \infty.\]
\end{theorem}
\begin{proof}[Sketch of proof:]
The proof closely follows \cite{And94} and \cite{Bren-Huisk17}. The crucial point is that the speed $G_{\rho, \kappa}$ satisfies a simple evolution equation:
\begin{align*}
(\partial_t - \dot \gamma_{\rho,\kappa}^{pq} \nabla_p \nabla_q) G_{\rho, \kappa} =  \dot \gamma_{\rho,\kappa}^{pq} A_{pq}^2 G_{\rho,\kappa} + \dot \gamma_{\rho,\kappa}^{pq} \bar R(e_p, \nu, e_q, \nu ) G_{\rho, \kappa}.
\end{align*}
Estimating crudely, one obtains by the maximum principle the bound
\[\min_{M_t} G_{\rho, \kappa} \geq \min_{M_0} G_{\rho, \kappa} \cdot e^{-Ct},\]
where $C$ depends on $N$ and an upper bound for the first derivatives of the speed. Hence the pinching condition 
\[\lambda_1 + \dots + \lambda_k > k \kappa\]
is preserved over finite time intervals. The pinching of the background curvature \eqref{eq:backg_pinch} ensures that the term 
\[\dot \gamma_{\rho,\kappa}^{pq} \bar R(e_p, \nu, e_q, \nu )\]
is nonnegative, so arguing as in Section \ref{sec:unif_parabolic} we conclude that $\max_{M_t} G_{\rho,\kappa}$ must become unbounded in finite time. The mean curvature of a solution satisfies 
\[(\partial_t - \dot \gamma_{\rho,\kappa}^{pq} \nabla_p \nabla_q) H \leq \dot \gamma_{\rho,\kappa}^{pq} A_{pq}^2 H + \ddot \gamma_{\rho,\kappa}^{pq,rs} \nabla^i A_{pq} \nabla_i A_{rs} + CH + C,\]
where $C$ depends only on an upper bound for $\dot \gamma_{\rho, \kappa}$ and $N$, and combining this with the evolution of $G_{\rho,\kappa}$ one obtains that $H/G_{\rho, \kappa}$ blows up at worst exponentially in time. This implies uniform parabolicity of the flow in appropriate coordinate systems. With these a priori estimates established, short-time existence and uniqueness are standard. The characterisation of singularity formation in terms of curvature blow-up follows from the regularity theory for fully nonlinear parabolic PDE (see for example \cite{Lang14}[Section 4.3]). 
\end{proof}

As in the Euclidean case, if $\rho$ is sufficiently small then solutions of \eqref{eq:CF_shifted} satisfy a convexity estimate:

\begin{theorem}
\label{thm:conv_est_N}
For each $n \geq 4$ and $3 \leq k \leq n-1$ there is a positive constant $\rho_0 = \rho_0(n,k)$ with the following property. Let $F : M\times [0,T) \to N$ be a compact solution of \eqref{eq:CF_shifted} with $\rho \leq \rho_0$. Then for every $\varepsilon >0$ there is a constant $C_\varepsilon = C_\varepsilon(n,k, \kappa, \rho, M_0, N)$ such that 
\[\lambda_1 \geq - \varepsilon G_{\rho,\kappa}- C_\varepsilon\]
on $M\times[0,T)$. 
\end{theorem}

This is proven by essentially the same argument as in the Euclidean case. Notice that the convexity estimate is trivially true at low curvature scales by $k$-convexity, whereas at high curvature scales the speed $G_{\rho, \kappa}$ is equal to $G_\rho$ up to small error terms depending on $\kappa$. Moreover, with the a priori estimates from Theorem \ref{thm:short_time_N} in place, all of the extra terms that enter the computations as a result of the curvature of $N$ are of lower order. In particular, the evolution of the second fundamental form is the same is in a Euclidean background, up to lower-order error terms:
\begin{align*}
|(\partial_t - \dot \gamma_{\rho,\kappa}^{pq} \nabla_p \nabla_q) A^i_j - \dot \gamma_{\rho,\kappa}^{pq} A_{pq}^2  A_j^i - \ddot \gamma_{\rho,\kappa}^{pq,rs} \nabla^i A_{pq} \nabla_j A_{rs}| \leq C|A| + C,
\end{align*}
where $C$ depends only on the initial data and $N$. Recall that we also made heavy use of the Codazzi equations, in Lemma \ref{lem:good_grad} and Lemma \ref{lem:lambda_1_grad_terms}, for example. Although $\nabla A$ is no longer totally symmetric, we have 
\[|\nabla_p A_{qr} - \nabla_r A_{pq}| \leq C,\]
where $C$ depends only on $N$. 

Using these facts, we prove cylindrical estimates and build the pinching functions $f_\eta$ in the same way as before, and choose $\rho$ small depending on $n$ and $k$ to make sure that the structure of the gradient terms in the evolution of $f_\eta$ is favourable. Any extra lower-order terms are then absorbed in the Stampacchia iteration procedure. In fact, Theorem \ref{thm:Stamp} goes through exactly as before, with the constants in the pinching estimate picking up extra dependencies on $\kappa$ and $N$. 

We note that the proof of the convexity estimate does not make further use of the background pinching condition \eqref{eq:backg_pinch} - the role of this assumption is only to force finite-time blow-up of solutions, which as we saw is required to ensure uniform parabolicity.

{\appendix

\section{Semiconvex functions}
\label{app:semiconvex}

In this section $(M,g)$ is a Riemannian manifold with volume element $d\mu_g$. 

\begin{definition}
We say that a function $f:M\to \mathbb{R}$ is locally semiconvex (resp. semiconcave) if for every $x \in M$ there is a positive radius $r >0$ such that $f$ is the sum of a smooth and a convex (resp. concave) function on $B_g(x,r)$. 
\end{definition}

Alexandrov's theorem (see Section 6.4 of \cite{Ev-Gar}) implies that a convex function on Euclidean space is almost-everywhere twice differentiable. Composition with a diffeomorphism in the domain preserves local semiconvexity, so choosing coordinates and applying Alexandrov's theorem we obtain:

\begin{lemma}
Let $f : M\to \mathbb{R}$ be locally semiconvex. Then there is a set of full measure in $M$ where $f$ has two derivatives.  
\end{lemma}

The distributional Hessian of a convex function can be interpreted as a Radon measure on its domain \cite{Ev-Gar}[Section 6.3]. This property carries over to locally semiconvex functions, and implies the following result:

\begin{lemma}
\label{lem:semiconvex_dist}
Let $f:M \to \mathbb{R}$ be a locally semiconvex function and consider a smooth vectorfield $X$ defined on $M$. Then there is a Radon measure $\mu_{X}[f]$ on $M$ such that 
\[\int_{M} \varphi d \mu_{X}[f] = - \int_{M} \nabla_p Y^{pq} \nabla_q f \, d\mu_g\]
holds for every compactly supported Lipschitz function $\varphi : M \to \mathbb{R}$, where $Y$ is the tensor $Y := \varphi X \otimes X$. Moreover, the density of the absolutely continuous part of $\mu_{X}[f]$ with respect ot $\mu_g$ is $\nabla^2 f(X,X)$. 
\end{lemma}
\begin{proof}
By approximation it suffices to consider the case that $\varphi$ is smooth. Choose a partition of unity $\{\zeta_i\}_{i=1}^N$ covering $\supp(\varphi)$, and let $\zeta_i$ be compactly supported in $U_i$. We may assume that each of the sets $U_i$ is equipped with coordinates $\{x^p\}$, and that $f$ is the sum of a smooth and a convex function on $U_i$. Then by Theorem 6.8 in \cite{Ev-Gar}, for each index $i$, and each pair of coordinate indices $p$ and $q$, there is a Radon measure $\chi_{pq}^i$ on $U_i$ with the property that
\[\int_{U_i} \zeta d \chi_{pq}^i = - \int_{U_i} \frac{\partial \zeta}{\partial x^p} \frac{\partial f}{\partial x^q} \, dx\]
for every $\zeta \in C^\infty_0(U_i)$. Moreover, the density of the absolutely continuous part of $\chi_{pq}^i$ is given by $\frac{\partial^2 f}{\partial x^p \partial x^q}$. That is, if we write $\hat \chi_{pq}^i$ for the singular part of $\chi_{pq}^i$, then 
\[\int_{U_i} \zeta \frac{\partial^2 f}{\partial x^p \partial x^q} dx + \int_{U_i} \zeta d \hat \chi_{pq}^i = - \int_{U_i} \frac{\partial \zeta}{\partial x^p} \frac{\partial f}{\partial x^q} \, dx\]
for every $\zeta \in C^\infty_0(U_i)$.
Applying this formula with 
\[\zeta = \zeta_i \varphi X^p X^q \sqrt{\det g},\]
one finds after a computation that 
\begin{align*}
\int_{U_i} \varphi \nabla^2 f(X,X) &\, d\mu_g + \int_{U_i} \zeta_i \varphi X^p X^q \sqrt{\det g}\, d \hat \chi_{pq}^i \\
& = \int_{U_i} \varphi X^p X^q \nabla_p f \nabla_q \zeta_i \, d\mu_g - \int_{U_i} \zeta_i \nabla_p Y^{pq} \nabla_q f \,d\mu_g.
\end{align*}
Summing over $i$ yields 
\begin{align*}
\int_{M} \varphi \nabla^2 f(X,X) &\, d\mu_g + \sum_{i} \int_{U_i} \varphi X^p X^q \sqrt{\det g}\, d \hat \chi_{pq}^i \\
& =  - \int_{M_t} \nabla_p Y^{pq} \nabla_q f \,d\mu_g,
\end{align*}
so it suffices to take 
\[\mu_{X}[f](U) := \int_U \nabla^2 f(X,X) d\mu_g + \sum_i \int_{U\cap U_i} \zeta_i X^p X^q \sqrt{\det g}\,d \hat \chi_{pq}^i\]
for each set $U \subset M$ which is measurable with respect to $\mu_g$. 
\end{proof}

From this we derive the following inequality, which is made use of in Section \ref{sec:conv_est}.

\begin{lemma}
Let $(M,g)$ be a compact Riemannian manifold. Suppose $f : M \to \mathbb{R}$ is semiconvex and let $\varphi: M\to \mathbb{R}$ be Lipschitz continuous and nonnegative. Let $T$ be a smooth positive-definite $(2,0)$-tensorfield. Then there holds 
\[\int_M \varphi T^{pq} \nabla_p \nabla_q f \, d\mu_g \leq - \int_M T^{pq} \nabla_p \varphi \nabla_q f \, d\mu_g - \int_M \varphi \nabla_p T^{pq} \nabla_q f\,d\mu_g.\]
\end{lemma}

\begin{proof}
By approximation it suffices to consider the case that $\varphi$ is smooth. Choose a partition of unity $\{\zeta_i\}_{i=1}^N$ on $M$, and let $\zeta_i$ be compactly supported in $U_i$. We may assume that $f$ is the sum of a smooth and a convex function on $U_i$, and that there is a local frame $\{\tilde \omega^p\}$ for the cotangent bundle of $M$ on $U_i$. Since $T$ is positive-definite and symmetric, we may apply the Gram-Schmidt algorithm to produce from $\{\tilde \omega^p\}$ a local frame $\{ \omega^p\}$ such that 
\[T(\omega^p, \omega^q) = \delta^{pq}.\]
We now define a local frame $\{ e_p\}$ for the tangent bundle by the condition
\[ \omega^p(  e_q) = \delta^p_q.\]
With respect to this basis 
\[T = T(\omega^p,  \omega^q)  e_p \otimes  e_q = \sum_p  e_p \otimes  e_p,\]
so we can express
\[T^{pq} \nabla_p \nabla_q f = \sum_p \nabla^2 f(e_p ,  e_p).\]
Since $\varphi \zeta_i$ is nonnegative,
\begin{align*}
\int_{M} \varphi \zeta_i T^{pq} \nabla_p \nabla_q f\, d\mu_g &= \sum_p \int_M \varphi \zeta_i \nabla^2 f ( e_p,  e_p)\, d\mu_g\\
&\leq \sum_p \int_M \varphi \zeta_i  \, d\mu_{e_p}[f],
\end{align*}
hence we can use Lemma \ref{lem:semiconvex_dist} to estimate
\begin{align*}
\int_{M} \varphi \zeta_i T^{pq} \nabla_p \nabla_q f\, d\mu_g &\leq   -\int_M \nabla_p Y_i^{pq} \nabla_q f\, d\mu_g
\end{align*}
where
\[Y_i := \varphi \zeta_i \sum_p  e_p \otimes  e_p  = \varphi \zeta_i T.\]
Expanding $\nabla_p Y_i^{pq} = \nabla_p(\varphi \zeta_i) T^{pq} + \varphi \zeta_i \nabla_p T^{pq}$ and summing over $i$ we arrive at 
\[\int_M \varphi T^{pq} \nabla_p \nabla_q f\, d\mu_g \leq - \int_M T^{pq} \nabla_p \varphi \nabla_q f\,d\mu_g - \int_M \varphi \nabla_pT^{pq} \nabla_q f\, d\mu_g.\]
\end{proof}
}

\bibliographystyle{alpha}
\bibliography{references}

\end{document}